\documentclass[a4paper,11pt]{article}     
\usepackage[italian, english]{babel} 
\usepackage{graphicx}
\usepackage{caption}
\usepackage{subcaption} 
\usepackage{float}   
\usepackage{fancybox}		  
\usepackage{makeidx} 
\usepackage{verbatim}
\usepackage{listings} % package for a colored Matlab code            
\usepackage{fancyvrb}
\usepackage{amsfonts}
\usepackage{color}
\usepackage{colortbl}
\usepackage{amsmath}
\usepackage{color}
\usepackage{dsfont}
\usepackage{epsfig}
\usepackage{verbatim}
\usepackage{listings}
\usepackage{cases}
\usepackage{epsfig}

\usepackage{amsfonts}
\usepackage{amssymb}
\usepackage{amsbsy}
\usepackage{amsmath}
\usepackage{enumerate}
\usepackage{stmaryrd}
\usepackage{mathtools}
\usepackage{bbm}
\usepackage{eucal}
\usepackage{pdfpages}
\usepackage{tikz}
\usepackage{bbm}
\usepackage{lipsum}
\usepackage{fancyhdr}
\usepackage{hyperref}
\usepackage{amsthm}

\DeclareMathOperator*{\argmin}{arg\,min}

\newtheorem{prop}{Proposition}%[section]
\newtheorem{thm}{Theorem}%[section]
\newtheorem{lemma}{Lemma}%[section]
\newtheorem{example}{Example}%[section]

\usepackage[margin=1in]{geometry}

\numberwithin{equation}{section}
\def\N {\mathbb{N}}

\def\R {\mathbb{R}}

\def\S {\mathbb{S}}

\def\L {\mathbb{L}}
\def\E {\mathbb{E}}
\def\D {\CMcal{D}}

\newcommand{\mres}{\mathbin{\vrule height 1.6ex depth 0pt width
0.13ex\vrule height 0.13ex depth 0pt width 1.3ex}}

\newcommand{\dd}{\,\mathrm{d}}

\newcommand\blfootnote[1]{%
  \begingroup
  \renewcommand\thefootnote{}\footnote{#1}%
  \addtocounter{footnote}{-1}%
  \endgroup
  }
\usepackage{fancyhdr}
\title{Analysis of Kinetic Models for Label Switching and Stochastic Gradient Descent}
\author{Martin Burger$^1$, Alex Rossi$^2$}
\date{$^{1, 2}$ Department of Mathematics, Friedrich--Alexander--Universität Erlangen--Nürnberg  \\[2ex] \today}
\begin{document}
\maketitle

\begin{abstract}
In this paper we provide a novel approach to the analysis of kinetic models for label switching, which are used for particle systems that can randomly switch between gradient flows in different energy landscapes. Besides problems in biology and physics, we also demonstrate that stochastic gradient descent, the most popular technique in machine learning, can be understood in this setting, when considering a time-continuous variant. 

Our analysis is focusing on the case of evolution in a collection of external potentials, for which we provide analytical and numerical results about the evolution as well as the stationary problem. 
    
\end{abstract}

\section{Introduction}
In this paper we introduce and analyze a kinetic equation that generalizes models for different natural phenomena:
\begin{equation}\label{fund_eq_meas0}
    \frac{\partial \rho}{\partial t} = \nabla \cdot(\rho \nabla E_s[\rho]) + K  (\bar{\rho}\otimes \mu - \rho ) 
\end{equation}
where $E_s$ is a collection of energy functionals labeled by a variable $s$, $\bar{\rho}$ is the marginal of $\rho$ on $\R^d$ integrating out $s$, $K > 0$ and $\mu$ a given probability measure on $S$. This problem models a process of label switching, where $K>0$ is the frequency of switching and $\mu$ describes the probability to switch to either variable. Between the switching events the particles follow a gradient flow according to the energy labelled with $s$. We will provide several examples of such processes from different fields below.

In our first analysis carried out in this paper, we are focusing on the case of $E_s$ being potential energies (with potentials $f(\cdot,s)$). Thus we arrive at the following equation for an evolving probability measure $\rho$ on $\R^d \times S$ with $S \subseteq \R^m$ a bounded label space, interpreted in weak sense:
\begin{equation}\label{fund_eq_meas}\frac{\partial \rho}{\partial t} = \nabla \cdot(\rho \nabla f(x, s)) + K  (\bar{\rho}\otimes \mu - \rho ).\end{equation}
 We specify that gradients and divergences are carried out just with respect to the variable $x$. We set the initial datum $\rho_0$ of \eqref{fund_eq_meas} as a probability measure on $\R^d \times S$, therefore the mass is preserved:
$$ \frac \dd{\dd t}\int_{\R^d \times S} \dd \rho_t (x, s) = 0.$$
Let us mention that, in the context of stochastic gradient descent, we are also interested in the case of $K$ being a time-dependent function with $K(t) \rightarrow \infty$ as $t \rightarrow \infty$, which corresponds to an asymptotically learning rate (inversely proportional to $K$).
\blfootnote{2020 Mathematics Subject Classification. Primary: 35A01, 35A02, 35R09, 45K05, Secondary: 68T07}
\blfootnote{Keywords and Phrases: Integro--partial differential equations, Label switching, Run--and--tumble particles, Stochastic gradient descent, Machine learning}

A particularly interesting case is when the measure $\rho$ is absolutely continuous with respect to $\L_d \otimes \mu$, where $\L_d$ denotes the Lebesgue measure on $\R^d$. As we shall see below, this is true if the initial value $\rho_0$ is absolutely continuous with respect to $\L_d \otimes \mu$. 
The corresponding density $\nu = \frac{d\rho}{d(\L_d \otimes \mu)}$ then satisfies the equation
\begin{equation}\label{fund_eq} \frac{\partial \nu}{\partial t} = \nabla \cdot(\nu \nabla f(x, s)) + K\int_S(\nu(x, s') - \nu(x, s))\dd \mu(s'),\end{equation}
with $\nu = \nu(t, x, s),\, t \geq 0,\, x \in \R^d,\, s \in S$.

Our main contributions in this paper are as follows:
\begin{itemize}
    \item We provide a detailed analysis of the transient and stationary problem for \eqref{fund_eq_meas}. In particular we establish the existence of stationary solutions under rather general conditions and discuss their structure depending on the properties of the potentials $f$ and their variation with $s$.
    
    \item We establish a novel approach to derive stability estimates in Wasserstein metrics via a coupling technique for the Kantorovich formulation, similar to the methodology in \cite{FP19}.
    
    \item We analyze the convergence of equation \eqref{fund_eq_meas} 
    
    \item Motivated by stochastic gradient descent, where $\frac{1}K$ is the learning rate typically decreasing to zero, we study the case of a time dependent coefficient $K$ with $K(t) \rightarrow \infty$ as $t\rightarrow \infty$ and show convergence 
\end{itemize}

As mentioned above,  \eqref{fund_eq_meas0} can be applied in many fields, hence we provide some further examples:
\begin{enumerate}
    \item Run--and--tumble processes, a general class of processes (see \cite{SCJ15} and \cite{SH19}), where the movement along straight lines is punctuated by random resets of the direction $\theta$. The main example is the locomotion of bacteria, whose density $\rho$ is described by an equation of type \eqref{fund_eq}:
    $$\partial_t \rho(r, \theta) = \nabla \cdot((v_se_z - vu(\theta))\rho(r, \theta)) - \alpha \rho(r, \theta) + \frac \alpha{2\pi}\int_{-\pi}^\pi \rho(r, \theta')\,\dd \theta'$$
    with the particular choices $S = (-\pi, \pi)$ for the label space (in this case, the set of the possible angles $\theta$), $\nabla f(x, y, z, \theta) = v_se_z - vu(\theta)$ for the force, where $vu(\theta)$ is the self--propulsion velocity, $-v_se_z = -\xi^{-1}\nabla V_{ext}$, with $V_{ext}(r) = \delta mgz$ external potential, and $\mu = \frac 1{2\pi}\L_1\,\mres (-\pi, \pi)$.
    \item Kinetic models with discrete label switching (cf. e.g. \cite{LT21}), which describe the density $\rho$ of a population, split into different categories $\CMcal{I} = \{1, \dots, N\}$ and the probability of switching at time $t$ from the label $y \in \CMcal{I}$ to the label $s \in \CMcal{I}$ is $T(t; s| y)$. In this case, the label space $\CMcal{I}$ is discrete and the problem is homogeneous in space, therefore $\rho$ is only label--dependent and satisfies the equation:
    $$\partial_t \rho(s, t) = \lambda \left(\int_\CMcal{I}T(t; s|y)\rho(y, t)\,\dd y - \rho(s, t)\right),$$
    for a certain interaction parameter $\lambda > 0$.
    \item Stochastic gradient descent processes, of which equation \eqref{fund_eq_meas} can be interpreted as a continuous--time version, since, as we shall prove, in hypothesis of convexity of the potential $f$, there is convergence of the solution of \eqref{fund_eq_meas} towards the minimizer of $f$. In order to generalize the classic discrete stochastic descent process, this approach is quite different from the well--established ones based on stochastic processes and stochastic differential equations (\cite{LAT21}, \cite{LL20}, \cite{LTW19}, \cite{WO21}, \cite{WOJ21}). In particular the continuous-time approximation we study maintains the finite speed of propagation of stochastic gradient descent and may have compactly supported stationary solutions (see the analysis below and its assumptions), while the cited diffusion approximations end up with infinite speed of propagation and hence unbounded support. 
    \item Diffusion with self-consistent forces and reaction-diffusion equations: equation \eqref{fund_eq_meas} can be interpreted as a variant with labels of some equations with self-consistent interactions such as the Poisson--Nernst--Planck equations with additional change of conformation or so-called interaction switching (cf. \cite{MD20}). %some size exclusion processes, presented in \cite{BDP10}, for the transport of particles in case the neighboring sites might be occupied. Such processes are relevant in molecular and cell biology. 
    A possible nonlinear generalization of \eqref{fund_eq_meas} can be provided by the material theory (cf. \cite{CCC13}), and in particular by models for colloids. In such systems, we have two species (\emph{big} and \emph{small}), whose interaction is described by the   nonlinear system:
    \begin{equation*} \begin{cases}
    \partial_t \rho_b(r, t) = - \nabla \cdot J_b + k_{sb}\rho_s(r, t) - k_{bs}\rho_b(r,t) \\
    \partial_t \rho_s(r, t) = - \nabla \cdot J_s + k_{bs}\rho_b(r, t) - k_{sb}\rho_s(r,t)
    \end{cases} \end{equation*}
    where $J_i \coloneqq -D_i[\nabla \rho_i + \rho_i\nabla\beta(u_i^{ext} + \mu_i^{ex})]$, $i = b, s$, are diffusive fluxes, $D_i \coloneqq k_BT/(3\pi\eta \sigma_i)$ are the diffusion constants of the two species, $u^{ext}_i$ are the confining external potentials, $\mu^{ex}_i(r,t)  \coloneqq \delta F_{ex}(\rho_i(r,t))/\delta \rho_i(r,t)$  are  the  functional  derivatives of  the  equilibrium  excess  free  energy  functional and $k_{bs}, k_{sb}$ are the switching rates between the two species. In this case, the term in the divergence is much more complicated and outside the scope of this article, whereas $S$ is composed just by two labels and $\mu$ is just a Bernoulli probability measure, rescaled by an adequate $K$.
\end{enumerate}
The structure of the paper is the following one: in section \ref{ex_un} we analyze the well--posedness of the evolution equation \eqref{fund_eq_meas}, in section \ref{scms} we introduce the coupling method, exploiting it to study the continuous dependency of solution to \eqref{fund_eq_meas} on $K$, $\mu$ and $\rho_0$ and long--time behaviour of solutions with strongly convex potentials. In section \ref{exis_ss} we deepen the existence of stationary states of equation \eqref{fund_eq_meas} and the shape of these states. In section \ref{conv_an}, we analyze the asymptotic behaviour of solutions when $K \to \infty$, also in the case $K = K(t)$ time--dependent. In section \ref{furt_per} there are some ideas about further developments for the theory on equations of type \eqref{fund_eq_meas}.
%%%%%%%%%%%%%%%%%%%%%%%%%%%%%%%%%%%%%%%%%%%%%%%%%%%%%%%%%%%%%%%%%%%%%%%%%%%%%%%%%%%%%%%%%%%%%%%%%
\section{Well--posedness of the Transient Problem} \label{ex_un}
In this section, we analyze the existence and uniqueness of solutions of equations \eqref{fund_eq_meas} and \eqref{fund_eq}, respectively. 
Throughout this section we will use the following important assumptions without further notice:
\begin{itemize}
\item $f \in C^{1,1}(\R^d \times S) \cap C^0(\R^d \times \bar S)$, where $C^{1,1}$ is the set of $C^1$ real functions with respect to the variable $x$.
\item $\nabla f \in Lip(\R^d \times S)$, i.e. there exists $L > 0$ such that $|\nabla f(x, s) - \nabla f(y, t)| \leq L|(x, s) - (y, t)|$, for every $x, y \in \R^d, s, t \in S$.
\end{itemize}
In particular, the last hypothesis and Rademacher's theorem (cf. theorem 11.49 in \cite{LE09}) imply that $\Delta f$ exists, but just $\L_d \otimes \mu$-a.e. and $\Delta f \in L^\infty(\R^d \times S)$. \\
In order to prove a theorem of existence and uniqueness for the solution of equation \eqref{fund_eq_meas}, we introduce the space $C^1([0,T], \CMcal{M}(\R^d \times S))$, $T > 0$, where $\CMcal{M}(\R^d \times S)$ are the Radon measures on $\R^d \times S$ with finite total variation, endowed with the norm:
\begin{align*}
\| \mu \|_{C^1([0,T], \CMcal{M}(\R^d \times S))} &\coloneqq \sup_{t \in [0, T]}\left(\|\mu(t)\|_{TV} + \|\mu'(t)\|_{TV}\right) = \\ &= \sup_{t \in [0, T]}\left(\sup_{\substack{\phi \in C_b(\R^d \times S) \\ \|\phi\|_{L^\infty} \leq 1}} \int_{\R^d \times S}\phi \,\dd \mu_t(x, s) + \sup_{\substack{\phi \in C_b(\R^d \times S) \\ \|\phi\|_{L^\infty} \leq 1}} \int_{\R^d \times S}\phi \,\dd \mu'_t(x, s)\right),
\end{align*}
which defines a Banach space (with $\mu'$ short-hand notation for $\frac{\partial \mu}{\partial t}$). We are also interested in the metric space $C^{0,1}([0, T], (\CMcal M^+(\R^d \times S), \|\cdot\|_{BL^*}))$ of Lipschitz continuous functions on $[0, T]$ with values in the positive Radon measures cone $\CMcal{M}^+(\R^d \times S)$, endowed with the norm:
\begin{align*}
\| \mu \|&_{C^{0,1}([0, T], (\CMcal M^+(\R^d \times S), \|\cdot\|_{BL^*}))} \coloneqq \\ &\sup_{t \in [0, T]}\sup_{\|\phi\|_{1, \infty} \leq 1} \int_{\R^d \times S}\phi \dd\mu_t(x, s) + \sup_{\substack{t, r \in [0, T]\\ t \neq r}}\frac{\sup_{\|\phi\|_{1, \infty} \leq 1} \int_{\R^d \times S}\phi \dd(\mu_t - \mu_r)}{|t - r|}
\end{align*}
with $\|\phi\|_{1, \infty} \coloneqq \|\phi\|_{L^\infty} + |\phi|_{Lip}$, where $|\cdot|_{Lip}$ is the Lipschitz constant. With the metric induced by this norm, $C^{0,1}([0, T], (\CMcal M^+(\R^d \times S), \|\cdot\|_{BL^*}))$ is a complete metric space. \\
We remind that, given generic measurable spaces $(X_1, \Sigma_1), (X_2, \Sigma_2)$, $F: X_1 \to X_2$ measurable and $\mu \in \CMcal{M}(X_1, \Sigma_1)$, the pushforward measure is $F\#\mu(A) \coloneqq \mu(F^{-1}(A))$, for every $A \in \Sigma_2$.
Before discussing the well-posedness of equation \eqref{fund_eq_meas}, we start with a short technical lemma.
\begin{lemma} \label{pushfwd}
If $G \in C^1([0, T] \times \R^d \times S,\, \R^d \times S)$ with $G = G(t; x, s)$, then the pushforward associated to $G$, $G\#: C^{0,1}([0, T], (\CMcal{M}^+(\R^d \times S), \|\cdot\|_{TV})) \to C^{0,1}([0, T], (\CMcal{M}^+(\R^d \times S), \|\cdot\|_{BL^*}))$ such that $G\#\mu(t) = G(t)\#\mu_t$, is well-defined. Moreover, if $\mu \in  C^{0,1}([0, T], (\CMcal{M}^+(\R^d \times S), \|\cdot\|_{TV}))$, $|G\#\mu|_{Lip} \leq \|G'\|_\infty\max_{t \in [0, T]}\|\mu_t\|_{TV} + |\mu|_{Lip}$.
\end{lemma}
\begin{proof}
We consider the $BL^*$ norm, defined as:
$$\|\nu\|_{BL^*} \coloneqq \sup_{\|\phi\|_{1, \infty} \leq 1}\left|\int_{\R^d \times S}\phi(x,s)\dd\nu(x, s)\right|$$
and we estimate the $BL^*$ norm of the pushforward of $\mu \in C^1([0, T], (\CMcal{M}^+(\R^d \times S), \|\cdot\|_{TV}))$ for $\phi$ such that $\|\phi\|_{1, \infty} \leq 1$ and $G(t) \coloneqq (X(t), S(t))$:
\begin{align*}
&\bigg|\int_{\R^d \times S}\phi\dd G\#\mu(t) - \int_{\R^d \times S}\phi\dd G\#\mu(r)\bigg| = \left|\int_{\R^d \times S}\phi(X(t), S(t))\dd\mu_t - \int_{\R^d \times S}\phi(X(r), S(r))\dd\mu_r\right| \\ &= \left|\int_{\R^d \times S}(\phi(X(t), S(t)) - \phi(X(r), S(r)))\dd\mu_t + \int_{\R^d \times S}\phi(X(r), S(r))\dd(\mu_t - \mu_r)\right|   \\
&\leq \|G'\|_\infty\|\mu_t\|_{TV}|t - r| + \|\mu_t - \mu_r\|_{TV} \leq \left(\|G'\|_\infty\max_{t \in [0, T]}\|\mu_t\|_{TV} + |\mu|_{Lip}\right)|t - r|
\end{align*}
and, therefore, passing to the supremum on $\phi$, $G\# \mu$ is a Lipschitz continuous function with constant less or equal than $\|G'\|_\infty\max_{t \in [0, T]}\|\mu_t\|_{TV} + |\mu|_{Lip}$. 
\end{proof}
We also introduce the space $\CMcal P(\R^d \times S) \subseteq \CMcal{M}^+(\R^d \times S)$ of probability measures on $\R^d \times S$ and more in particular $\CMcal P_2(\R^d \times S)$, the space of probability measures with finite second moment.
\begin{thm}\label{ex_thm}For every initial datum $\rho_0 \in \CMcal{P}_2(\R^d \times S)$ and every $T \geq 0$, there exists a unique distributional solution $\rho \in C^{0,1}([0, T], (\CMcal{M}^+(\R^d \times S), \|\cdot\|_{BL^*}))$ of equation \eqref{fund_eq_meas} such that $\rho_t \in \CMcal{P}_2(\R^d \times S)$ for every $t \in [0, T].$ \end{thm}
\begin{proof} We consider a sequence $(f^n)_n$ of $C^2$ functions such that $\nabla f^n$ is $L$--Lipschitz continuous for every $n$, and $\nabla f^n \to \nabla f$ uniformly. First of all, we find a solution to equation \eqref{fund_eq_meas} with $f^n$ instead of $f$ with initial datum $\rho_0$. We fix $n \in \N$, exploit the method of characteristics and we write a characteristic line $(X, S)$ such that 
\begin{equation}\label{charct_eq}\dot X(t) = - \nabla f^n(X(t), S(t)), \qquad \dot S(t) = 0,
\end{equation}
with $X(0) = x \in \R^d$, $S(0) = s \in S$. This characteristic line is well-defined and $C^1$-regular, since $\nabla f^n$ is Lipschitz continuous. Then we have that 
$$\frac \dd{\dd t} (X_t, S_t) \#\rho_t^n = \partial_t \left((X_t, S_t)\#\rho_t^n\right) - (X_t, S_t)\#(\nabla \rho_t^n \cdot \nabla f^n),$$ where $(X_t, S_t)\# \rho_t^n$ indicates the pushforward given by the characteristic $(X_t, S_t)$ applied to the measure $\rho_t^n$. We apply the product rule 
$$\nabla \cdot(\rho^n\nabla f^n) = \nabla\rho^n \cdot \nabla f^n + \rho^n\Delta f^n$$ and we can substitute the previous expression into the pushforwarded version of equation \eqref{fund_eq_meas} to get:
\begin{equation}\label{ode} \frac \dd{\dd t}(X_t, S_t)\#\rho_t^n = (X_t, S_t)\#\left((\Delta f^n - K)\rho_t^n\right) + K(X_t, S_t)\# (\bar \rho_t^n \otimes \mu) \eqqcolon A(t, (X_t, S_t)\#\rho_t^n). \end{equation}
The operator $A: [0, T] \times \CMcal M(\R^d \times S) \to \CMcal M(\R^d \times S)$ is well defined, continuous in the first variable and Lipschitz continuous in the second one, since $\Delta f^n$ are continuous and uniformly bounded. Therefore, we can solve equation \eqref{ode} for $(X_t, S_t)\#\rho^n$, where from the Picard-Lindelöf theorem we obtain that there exists a unique $\eta^n \in C^1([0, T], \CMcal M(\R^d \times S))$ which solves 
\begin{equation}\label{ode2}
\frac \dd{\dd t} \eta_t^n = A(t, \eta_t^n) = (\Delta f^n(X(t), S(t)) - K)\eta_t^n + K\bar \eta_t^n \otimes \mu.
\end{equation}
Moreover, $\eta^n$ satisfies $\eta_t^n = e^{tA} \eta_0$, where $e^{tA}$ is the semigroup associated to $A$, therefore, for every positive initial datum $\eta_0$, $\eta_t^n \in \CMcal M^+(\R^d \times S)$. It holds:
$$\Delta f^n(X(t), S(t))\eta_t^n = \nabla \cdot(\eta_t^n\nabla f^n(X(t), S(t))) - \nabla \eta_t^n \cdot \nabla f^n(X(t), S(t)),$$
and therefore:
$$\frac{\dd}{\dd t}\eta_t^n + \nabla\eta_t^n \cdot \nabla f^n(X(t), S(t)) = \nabla \cdot(\eta_t^n\nabla f^n(X(t), S(t))) - K(\bar \eta_t^n \otimes \mu - \eta_t^n).$$
One can notice that, if $\eta_0 \coloneqq \rho_0$, this is the pushforwarded version of \eqref{fund_eq_meas} with respect to $(X_t, S_t)$, which is the unique solution of the characteristic equation \eqref{charct_eq}, gaining that for every $n \in \N$ there exists a unique solution $\rho^n \in C^{0,1}([0, T], (\CMcal M^+(\R^d \times S), \|\cdot\|_{BL^*}))$, such that:
\begin{equation}\label{eq_n} \frac{\partial \rho^n}{\partial t} = \nabla \cdot(\rho^n \nabla f^n(x, s)) + K  (\bar{\rho^n}\otimes \mu - \rho^n ), \end{equation}
which we interpret as a weak formulation of the equation. This is due to lemma \ref{pushfwd} applied to the reversed characteristic $G(t) \coloneqq (X(R - t), S(R - t))$ with initial datum $(X(R), S(R))^{-1}$ as $R$ varies in $[0, T]$, which is regular, thanks to the theorems on differentiability with respect to the initial data. In particular, if we integrate \eqref{eq_n} against a constant function, we obtain that $\frac \dd{\dd t}\int_{\R^d \times S}\,\dd\rho_t^n(x, s) = 0$, which implies conservation of mass and therefore $\rho_t^n \in \CMcal P(\R^d \times S)$ for every $t \geq 0$ and $n \in \N$.
For every $t \geq 0$, this solution $\rho_t^n \in \CMcal{P}_2(\R^d \times S)$:
\begin{align*}\frac \dd{\dd t} &\int_{\R^d \times S}|x|^2\,\dd\rho_t^n(x,s) = -2\int_{\R^d \times S}x \cdot \nabla f^n(x, s)\,\dd\rho_t^n(x, s) = \\
&= -2\int_{\R^d \times S}(x - 0)\cdot(\nabla f^n(x, s) - \nabla f^n(0, s))\,\dd \rho_t^n(x, s) - 2\int_{\R^d \times S}x\cdot \nabla f^n(0, s)\,\dd\rho_t^n(x, s) \leq \\ &\leq 2L\int_{\R^d \times S}|x|^2\, \dd\rho_t^n(x, s) + 2C\int_{\R^d \times S}|x|\,\dd \rho_t^n(x, s) \leq \\ &\leq (2L + 1)\int_{\R^d \times S}|x|^2\,\dd \rho_t^n(x, s) + C,
\end{align*}
since $|\nabla f(0, s)| \leq C$, because $\nabla f(0, s)$ is continuous (and, therefore, bounded) on $\bar S$. Applying the Grönwall lemma, we obtain that for $t \in [0, T]$:
\begin{equation} \label{bound_2}
 \int_{\R^d \times S}|x|^2\,\dd\rho_t^n(x,s) \leq \left(\int_{\R^d \times S}|x|^2\,\dd\rho_0(x,s)\right)e^{(2L + 1)t} + C\frac{e^{(2L + 1)t} -1}{2L + 1}.
 \end{equation}
Considering that $S$ is bounded, $\int_{\R^d \times S}|s|^2\,\dd\rho_t^n(x, s)$ is furthermore uniformly bounded in time. The sequence $(\rho_t^n)_n$ is hence tight and by the Prokhorov theorem (theorem $5.13$ in \cite{C11}), we can select a subsequence (without changing the notations) weakly convergent to $\rho_t \in \CMcal P_2(\R^d \times S)$ for every $t \in [0, T]$, thanks to the uniform bound \eqref{bound_2}. \\ 
We rewrite \eqref{eq_n} in weak form for $\phi \in C_b(S, \CMcal D(\R^d))$, that is the space of bounded continuous functions on $S$ and $C^\infty$ with compact support on $\R^d$, for every $t \in [0, T]$:
\begin{equation*}\begin{aligned}  \label{first_eq}
\int_{\R^d \times S}\phi\,\dd\rho_t^n - &\int_{\R^d \times S}\phi\,\dd\rho_0 = \\&= -\int_0^t\int_{\R^d \times S}(\nabla\phi \cdot \nabla f^n +  K\phi)\,\dd\rho_s^n\,\dd s + \int_0^t\int_{\R^d \times S}K\phi \,\dd(\bar \rho_s^n \otimes \mu)\,\dd s
\end{aligned}\end{equation*}
and, since $\int \nabla \phi\cdot \nabla f^n\,\dd \rho_s^n = \int \nabla \phi\cdot (\nabla f^n - \nabla f)\,\dd \rho_s^n + \int \nabla \phi\cdot \nabla f\,\dd \rho_s^n$, by weak convergence:
\begin{equation}\begin{aligned} \label{second_eq}
\int_{\R^d \times S}\phi\,\dd\rho_t - &\int_{\R^d \times S}\phi\,\dd\rho_0 = \\ &= -\int_0^t\int_{\R^d \times S}(\nabla\phi \cdot \nabla f +  K\phi)\,\dd\rho_s\,\dd s + \int_0^t\int_{\R^d \times S}K\phi \,\dd(\bar \rho_s \otimes \mu)\,\dd s,
\end{aligned}\end{equation}
thanks to the Lebesgue theorem, because the involved functions are bounded on $[0, T]$. Furthermore, thanks to \eqref{second_eq}, $\rho \in C^0([0, T], (\CMcal M^+(\R^d \times S), \|\cdot\|_{BL^*}))$. Equation \eqref{ode2} tells us that the Lipschitz constants of $\eta^n$ are uniformely bounded and therefore from lemma \ref{pushfwd} it follows that also $\rho \in C^{0,1}([0, T], (\CMcal{M}^+(\R^d \times S), \|\cdot\|_{BL^*}))$. This shows existence of a weak solution to equation \eqref{fund_eq_meas}.  \\
To study the uniqueness, we write an adjoint version of \eqref{fund_eq_meas}:
\begin{equation} \label{fadj_eq}
    \frac{\partial \phi}{\partial t} = - \nabla \phi \cdot \nabla f + K\int_S \phi(x, s')\,\dd\mu(s') - K\phi + \psi
\end{equation}
for which we can repeat the argument with the characteristic lines as in the first part of the proof. Therefore, for every initial datum there exists a unique solution $\phi \in C^{0,1}([0, T], C_b(\R^d \times S))$ for each continuous and bounded right-hand side $\psi$. Assuming that $\rho^1$ and $\rho^2$  are two weak solutions of \eqref{fund_eq_meas} with the same initial datum, a standard adjoint argument by integrating \eqref{fadj_eq} against $\rho_1$ and $\rho_2$ shows, thanks to \eqref{second_eq}, that 
$$ \int_0^t \int_{\R^d \times S} \psi \, \dd (\rho^1 - \rho^2) \, \dd r = 0 $$
for each continuous function $\psi$ and $t > 0$, hence $\rho^1 = \rho^2$.
%\begin{equation}\begin{aligned}\label{ode_sol} \rho(t, x(t), s) = & e^{\int_0^t (\Delta f(x(u),s) - K)\,\dd u }\rho_0(x, s) + \\ &+ K\int_0^t e^{\int_r^t(\Delta f(x(u), s) - K)\,\dd u}\int_S \rho(r, x(r), s')\,\dd \mu(s')\dd r. \end{aligned}\end{equation}
%Thanks to the continuity in the variable $t$, for fixed $x(0) \in \R^d$, $s \in S$ a.e., we have positivity for $\rho(\cdot, x(\cdot), s)$ in a neighbourhood of zero and using \eqref{ode_sol} positivity holds. Indeed, if we consider a point $0 < \bar T < T$ for which, for fixed $x(0), s$, $\rho(t, x(t), s) > 0$ when $0 \leq t < \bar T$ and $\rho(\bar T, x(\bar T), s) = 0$, then the first term is nonnegative, as well as the second since $0 \leq r \leq \bar T$.
\end{proof}
Let us introduce also the space $C^1([0,T], L^1(\R^d \times S))$, $T > 0$, where the integration on $\R^d \times S$ is always performed with respect to the product measure $\L_d \otimes \mu$, endowed with the following norm:
$$ \| u \|_{C^1([0,T], L^1(\R^d \times S))} \coloneqq \sup_{t \in [0, T]}\left(\|u(t)\|_{L^1(\R^d \times S)} + \|u'(t)\|_{L^1(\R^d \times S)}\right).$$
\begin{thm}For every initial datum $\nu_0 \in \CMcal{P}_2(\R^d \times S)$ such that $\nu_0 \ll \L_d \otimes \mu$ and every $T > 0$, there exists a unique solution $\nu \in C^{0,1}([0, T], (\CMcal{M}^+(\R^d \times S), \|\cdot\|_{BL^*}))$ for equation \eqref{fund_eq_meas} such that $\nu_t \in \CMcal{P}_2(\R^d \times S)$ and $\nu_t \ll \L_d \otimes \mu$ for every $t \in [0, T]$. This means that, if $\rho_0 \ll \L_d \otimes \mu$ in equation \eqref{fund_eq_meas}, then $\rho_t \ll \L_d \otimes \mu$ for every $t \in [0, T]$ and equation \eqref{fund_eq_meas} is reduced to \eqref{fund_eq}. 
\end{thm}
\begin{proof}Since $\|f\,\L_d \otimes \mu\|_{TV} = \|f\|_{L^1}$, one can repeat the same reasoning of theorem \ref{ex_thm} with norm $\|\cdot \|_{C^1([0,T], L^1(\R^d \times S))}$.
\end{proof}
%%%%%%%%%%%%%%%%%%%%%%%%%%%%%%%%%%%%%%%%%%%%%%%%%%%%%%%%%%%%%%%%%%%%%%%%%%%%%%%%%%%%%%%%%%
\section{Coupling method and stability} \label{scms}
In order to investigate the stability of equation \eqref{fund_eq_meas} with respect to changes of the parameters, we shall follow the approach exposed in \cite{FP19}. If $\mu, \nu \in \CMcal P_2(\R^d \times S)$, we introduce their $2$-Wasserstein distance by:
$$W_2(\mu, \nu)^2 \coloneqq \inf_{\xi \in Coup(\mu, \nu)}\int_{\R^{2d} \times S^2}|x_1 - x_2|^2 + |s_1 - s_2|^2\dd\xi(x_1, x_2, s_1, s_2)$$
where $Coup(\mu, \nu)$ is the set of \emph{couplings} between $\mu$ and $\nu$, i.e. the set of $\xi \in \CMcal{P}(\R^{2d}\times S^2)$ such that the marginals of $\xi$ on $\R^d \times S$ are $\mu$ and $\nu$. Similar definitions hold for Wasserstein distances between probability measures on different spaces (for instance, if $\mu, \nu \in \CMcal{P}_2(S)$) and, for the sake of simplicity, we will use the same notation. \\
Let $\rho_1, \rho_2$ be solutions of our equation \eqref{fund_eq_meas} with initial data $\rho_{0, 1}$ and $\rho_{0, 2}$ and probability measures $\mu_1$ and $\mu_2$ in equation \eqref{fund_eq_meas}. Let $\Pi_0$ be a coupling between $\rho_{0, 1}$ and $\rho_{0, 2}$, initial datum of the solution $\Pi$ of the measure-valued coupling equation:
\begin{equation} \begin{aligned}\label{coup_eq}
     \frac{\partial \Pi_t}{\partial t}  =& \, \nabla_{x_1} \cdot(\Pi_t \nabla_{x_1} f(x_1,s_1)) +  \nabla_{x_2} \cdot(\Pi_t \nabla_{x_2} f(x_2,s_2)) 
     + \\ & + K~ \int_S \int_S ( \Gamma(s_1,s_2) \Pi_t(x_1,\dd s_1',x_2,\dd s_2') - \Pi_t(x_1,s_1,x_2, s_2))%\dd s_1' \dd s_2',
 \end{aligned} \end{equation}
 where $\Gamma \in \CMcal{P}(S^2)$ is the optimal coupling for the quadratic cost between $\mu_1$ and $\mu_2$, that is:
 $$\Gamma \coloneqq \argmin_{\gamma \in Coup(\mu_1, \mu_2)}\, \int_{S^2}|s_1 - s_2|^2\,\dd \gamma(s_1, s_2).$$
 If, for example, $\mu_1 = \mu_2$, then $\Gamma$ has the two marginals equal to $\mu_1$ and support on the diagonal of $S^2$ ($\Gamma(diag(S^2)) = 1$), that is $\Gamma = \mu_1\delta_{s_1 - s_2}$. We refer to \cite{SAN15} for further information.
 \begin{lemma}For every initial datum $\Pi_0 \in \CMcal{P}_2(\R^{2d} \times S^2)$ and $T > 0$, there exists a unique solution $\Pi \in C^{0,1}([0, T], (\CMcal M^+(\R^{2d} \times S^2), \|\cdot\|_{BL^*}))$ to equation \eqref{coup_eq} such that $\Pi_t \in \CMcal P_2(\R^{2d} \times S^2)$ for every $t \in [0, T]$.
 \end{lemma}
 \begin{proof} This equation is well--posed thanks to the fact that $\nabla_{x_1} \cdot(\Pi_t \nabla_{x_1} f(x_1,s_1)) +  \nabla_{x_2} \cdot(\Pi_t \nabla_{x_2} f(x_2,s_2)) = \nabla_x \cdot(\Pi_t \nabla_xg(x,s))$, where $x \coloneqq (x_1, x_2), s \coloneqq (s_1, s_2)$ and $g(x,s) \coloneqq f(x_1, s_1) + f(x_2, s_2)$ and hence $\Pi$ solves an equation of type \eqref{fund_eq_meas} with $\mu = \Gamma$ and we can apply theorem \ref{ex_thm}.
 \end{proof}
In particular, with our choice of $\Pi_0$, the marginals of $\Pi$ are $\rho_1$ and $\rho_2$ by uniqueness of the solution of \eqref{fund_eq_meas}.
 \begin{thm} \label{thm_stab}
Let $\rho_1$ be the solution of \eqref{fund_eq_meas} with initial datum $\rho_{0, 1} \in \CMcal{P}_2(\R^d \times S)$ and probability measure $\mu_1 \in \CMcal P(S)$ and $\rho_2$ the solution of \eqref{fund_eq_meas} with initial datum $\rho_{0, 2} \in \CMcal{P}_2(\R^d \times S)$ and probability measure $\mu_2 \in \CMcal P(S)$, i.e.
$$ \frac{\partial \rho_i}{\partial t} = \nabla \cdot(\rho_i \nabla f(x, s)) + K\int_S(\rho_i(x, s') - \rho_i(x, s))\dd \mu_i(s'),  \qquad i=1,2. $$
Then the following stability estimate holds:
  \begin{equation}W_2(\rho_1(t), \rho_2(t)) \leq e^{\frac 32 Lt}\,W_2(\rho_{0, 1}, \rho_{0, 2}) + \sqrt{\frac{Ke^{3Lt} - K}{3L}}W_2(\mu_1, \mu_2). \label{stab_ineq} \end{equation}
 \end{thm}
 \begin{proof} Since $\Pi$ is solution of the coupling equation \eqref{coup_eq}, by previous definition of $x$ and $s$:
\begin{align*}
\frac{\dd}{\dd t} &\int_{(\R^d)^2} \int_{S^2} (|x_1 - x_2|^2 + |s_1-s_2|^2) \dd\Pi_t(x_1,x_2,s_1,s_2) = \\
& =\int_{(\R^d)^2} \int_{S^2}\left(|x_1 - x_2|^2 + |s_1-s_2|^2\right)\Big(\nabla_{x_1} \cdot(\Pi_t \nabla_{x_1} f(x_1,s_1)) +  \nabla_{x_2} \cdot(\Pi_t \nabla_{x_2} f(x_2,s_2)) +
   \\  &\qquad K\int_{S^2}\left( \Gamma(\dd s_1,\dd s_2)\Pi_t(\dd x_1,\dd s_1',\dd x_2,\dd s_2')  - \Pi_t(\dd x_1,\dd s_1,\dd x_2,\dd s_2)\dd s_1' \dd s_2'\right)\Big)
   = \\
   &= \int_{(\R^d)^2} \int_{S^2}\Big( 2(x_1 - x_2)\cdot \left(\nabla_{x_2}f(x_2, s_2)- \nabla_{x_2}f(x_2, s_1) + \nabla_{x_2}f(x_2, s_1) - \nabla_{x_1}f(x_1, s_1)\right)\Pi_t +  \\
   &\qquad K\left(|x_1 - x_2|^2 + |s_1-s_2|^2\right)\int_{S^2}\left( \Gamma(\dd s_1,\dd s_2) \Pi_t(\dd x_1,\dd s_1', \dd x_2, \dd s_2')  - \Pi_t(\dd x_1,\dd s_1, \dd x_2, \dd s_2)\dd s_1'\dd s_2' \right)\Big)  \\ &\leq
   \int_{(\R^d)^2} \int_{S^2}\Big(2 L|x_1 - x_2|(|x_1 - x_2| + |s_1 - s_2|) + K|x_1 - x_2|^2 -  K|x_1 - x_2|^2 - \\ &\qquad - K|s_1 - s_2|^2\Big) \dd\Pi_t(x_1,x_2,s_1,s_2) + K\int_{S^2}|s_1 - s_2|^2\,\dd \Gamma(s_1,s_2) 
   \\ &\leq  \int_{(\R^d)^2} \int_{S^2} \Big(3L|x_1 - x_2|^2 + L|s_1 - s_2|^2  - K|s_1 - s_2|^2\Big) \dd\Pi_t(x_1,x_2,s_1,s_2) + \\ &\qquad  K \int_{S^2}|s_1 - s_2|^2\,\dd \Gamma(s_1,s_2)
   \\ &\leq 3L\int_{(\R^d)^2} \int_{S^2} (|x_1 - x_2|^2 + |s_1-s_2|^2)\, \dd\Pi_t(x_1, s_1, x_2, s_2) +  K \int_{S^2}|s_1 - s_2|^2\,\dd\Gamma(s_1, s_2).
 \end{align*}
 
 where we used integration by parts and the Young inequality. Exploiting the previous inequality and the optimal transport properties of $\Gamma$, we apply the Grönwall lemma and it holds that:
 \begin{equation}\begin{aligned}\label{wass_ineq}&\int_{(\R^d)^2} \int_{S^2}(|x_1 - x_2|^2 + |s_1-s_2|^2)\, \dd\Pi_t(x_1, s_1, x_2, s_2) 
  \\ &\leq e^{3Lt}\int_{(\R^d)^2} \int_{S^2} (|x_1 - x_2|^2 + |s_1-s_2|^2)\,  \dd\Pi_0(x_1, s_1, x_2, s_2) + K\frac{e^{3Lt} - 1}{3L}W_2^2(\mu_1, \mu_2),\end{aligned} \end{equation}
  where $W_2(\cdot, \cdot)$ is the quadratic Wasserstein metric.  Then, taking in \eqref{wass_ineq} the infimum on the possible couplings between $\rho_1$ and $\rho_2$ and considering that $\Pi_0$ can be any possible coupling of the initial data, we obtain the squared version of \eqref{stab_ineq}.
 \end{proof}
 Let us mention that indeed the exponent of the exponentially growing term in time in theorem \ref{thm_stab} can be optimized a bit further depending on $\frac{K}L$, but here we chose to give the more simpler statement with an exponent independent of $K$ instead. 
 
 With the same notations of theorem \ref{thm_stab}, we can state a stronger condition in case of strongly convex potentials.
 \begin{thm} \label{thm_conv}If $f$ is strongly convex uniformly in $s \in S$, that is $\exists m > 0$ such that $(\nabla f(x, s) - \nabla f(y, s))\cdot(x - y) \geq m|x - y|^2$ for all $x, y \in \R^d$ and $s \in S$, then  for every $\delta > 0$ $\exists \alpha > 0$ such that the following inequality holds:
  \begin{equation}\label{imp_ineq} W_2(\rho_1(t), \rho_2(t)) \leq \frac{\sqrt{\max\{1, \alpha\}}e^{-\frac c 2 t}}{\sqrt{\min\{1, \alpha\}}}W_2(\rho_{0,1}, \rho_{0,2}) + \sqrt{\alpha K\frac{1 - e^{-ct}}{c\min\{1, \alpha\}}}W_2(\mu_1, \mu_2), \end{equation}
 with $c = \min\{2m - \delta, K - \delta\}$. Moreover, the stationary state $\rho_\infty$ of \eqref{fund_eq_meas} is unique and the following asymptotic estimate holds:
 \begin{equation}\label{conv_stat} W_2(\rho(t), \rho_\infty) \leq \frac{\sqrt{\max\{1, \alpha\}}e^{-\frac c2 t}}{\sqrt{\min\{1, \alpha\}}}W_2(\rho_0, \rho_\infty), \end{equation}
 where $\rho$ is solution to \eqref{fund_eq_meas} with initial datum $\rho_0$; namely the convergence towards the stationary state is exponential.
If, in particular, $f$ has a unique minimizer $x^*$ for every $s \in S$, hence $\rho_\infty = \delta_{x^*} \otimes \mu$.
 \end{thm}
 \begin{figure}
    \centering
    \begin{subfigure}[b]{0.49\textwidth}
        \centering
        \includegraphics[width=\textwidth]{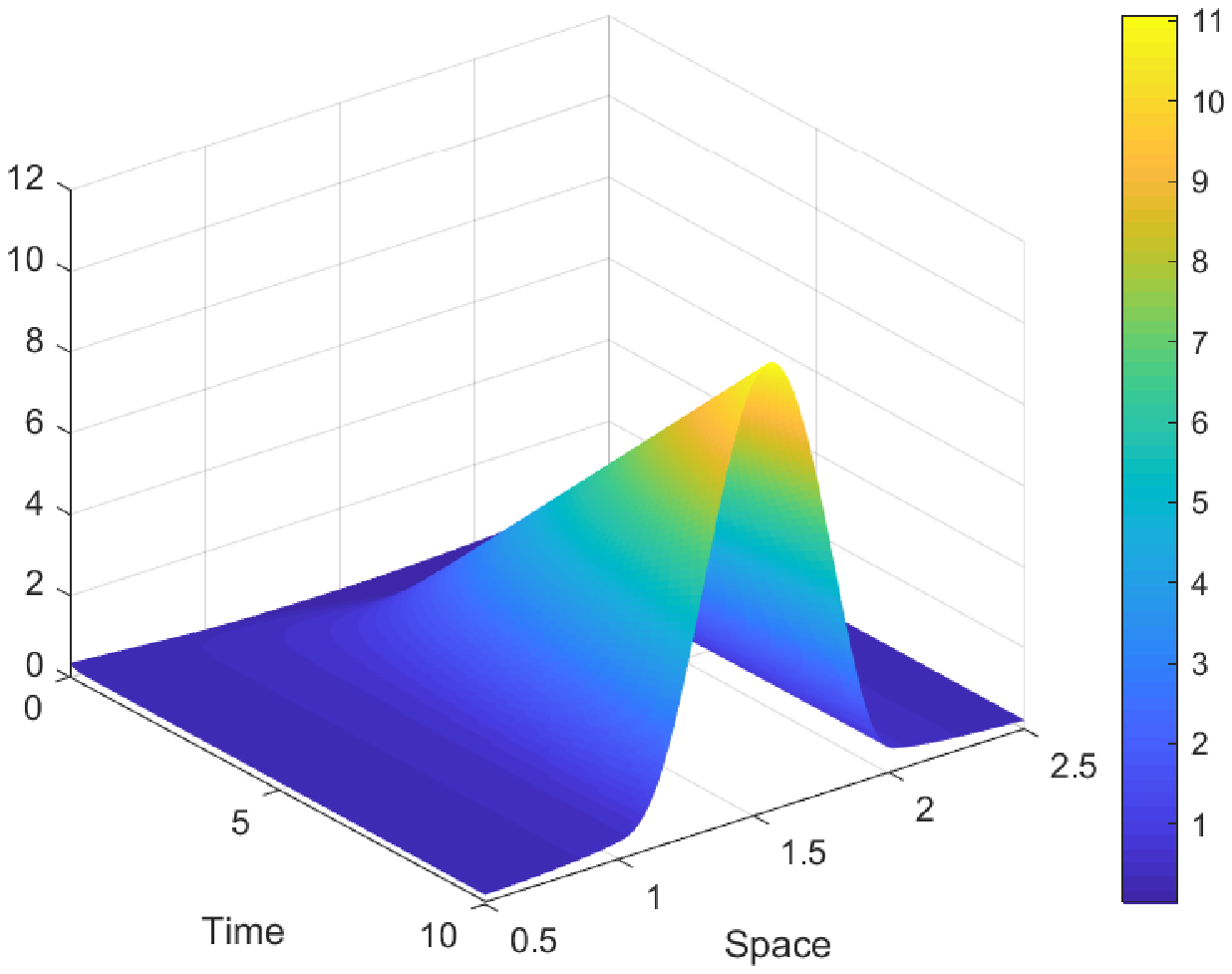}
    \end{subfigure}
    \begin{subfigure}[b]{0.49\textwidth}
        \centering
        \includegraphics[width=\textwidth]{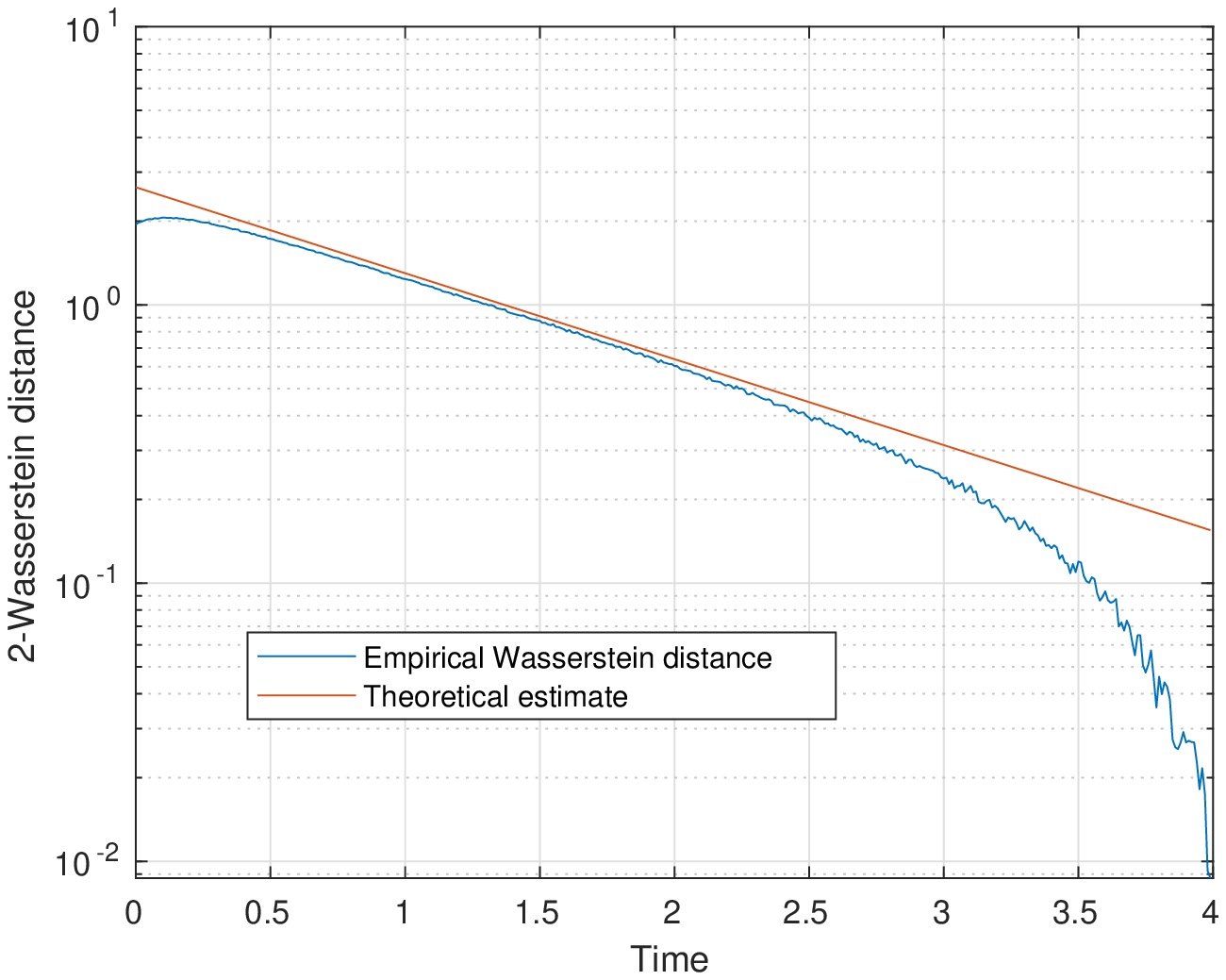}
    \end{subfigure}
    \begin{subfigure}[b]{0.49\textwidth}
        \centering
        \includegraphics[width=\textwidth]{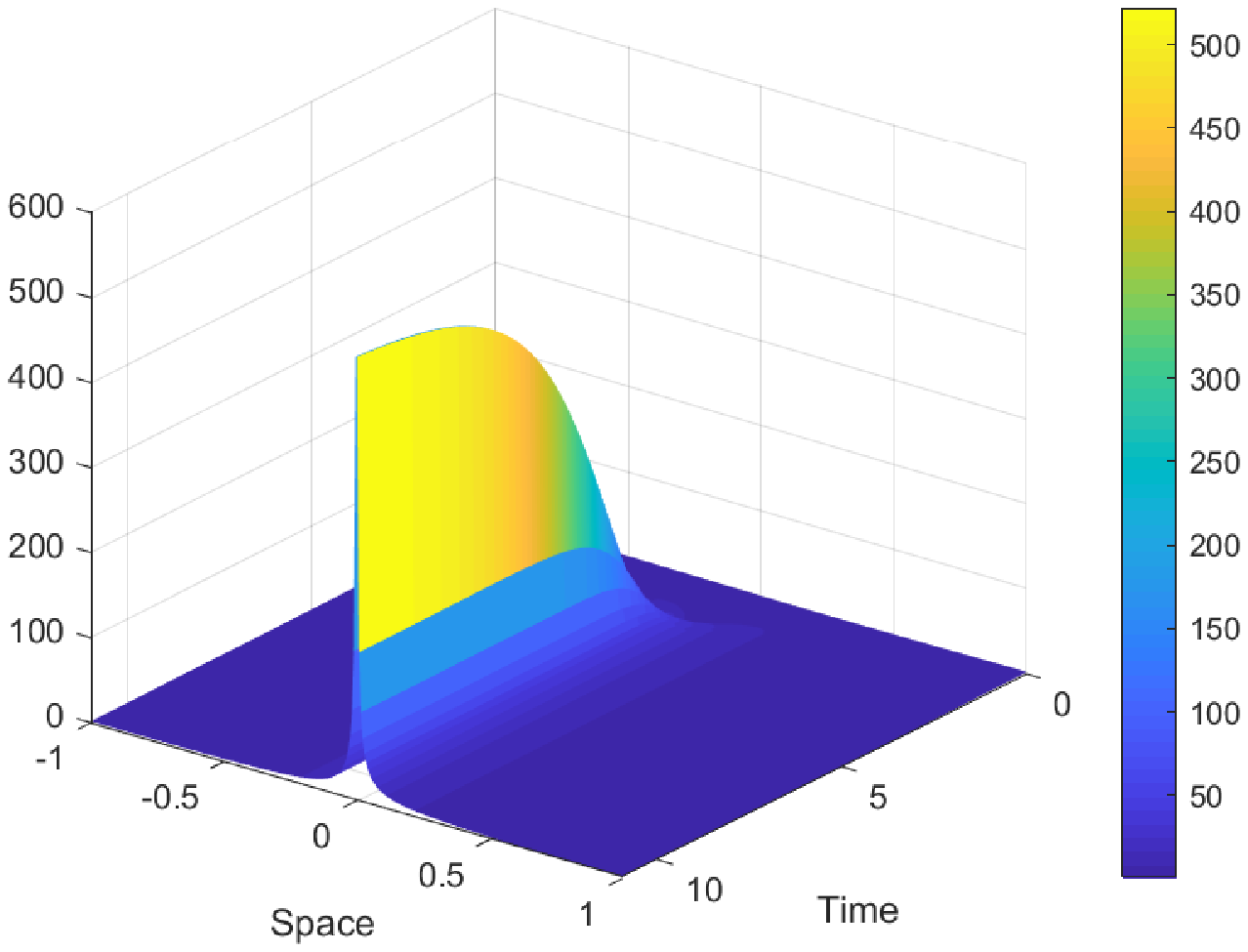}
    \end{subfigure}
     \begin{subfigure}[b]{0.49\textwidth}
        \centering
        \includegraphics[width=\textwidth]{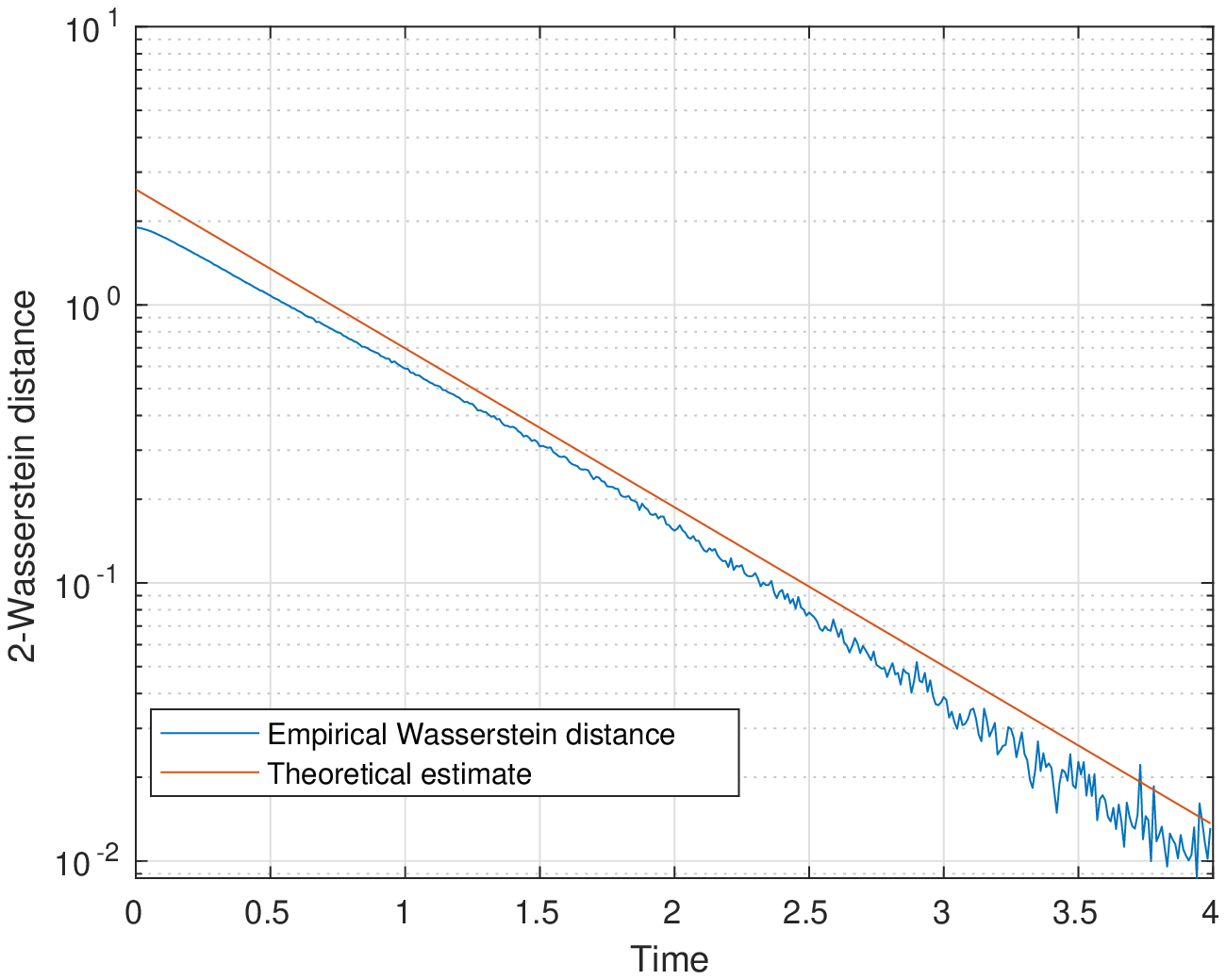}
    \end{subfigure}
    \caption{Solutions with fixed label and convergence in semilogarithmic scale towards stationary states given by $d = 1$ with quadratic potentials $f(x, s) = \frac s2|x - s|^2$ and $f(x, s) = sx^2$, $S = \{1, 2\}$ and $\mu = Bern(p)$ compared with the estimate provided by theorem \ref{thm_conv}.}
    \label{conv_ss}
\end{figure}
 \begin{proof}
 We can repeat the calculation of theorem \ref{thm_stab} by the introduction of a weight $\alpha > 0$, getting the estimate:
 \begin{align*} \frac{\dd}{\dd t} &\int_{(\R^d)^2} \int_{S^2} (|x_1 - x_2|^2 + \alpha|s_1-s_2|^2)\,\dd \Pi_t(x_1,x_2,s_1,s_2)   \\
&\leq \int_{(\R^d)^2} \int_{S^2} \Big(2(x_1 - x_2)\cdot \left(\nabla_{x_2}f(x_2, s_2)- \nabla_{x_2}f(x_2, s_1)\right) - 2m|x_1 - x_2|^2 - \alpha K|s_1 - s_2|^2\Big)\cdot \\ &\qquad \cdot \dd\Pi_t(x_1,x_2,s_1,s_2) + \alpha K \int_{S^2}|s_1 - s_2|^2\dd\Gamma(s_1,s_2)  \\
&\leq \int_{(\R^d)^2} \int_{S^2} \Big((\epsilon L-2m)|x_1 - x_2|^2 + \left(\frac L\epsilon - \alpha K\right)|s_1 - s_2|^2\Big)\dd\Pi_t(x, s) + \alpha K \int_{S^2}|s_1 - s_2|^2\dd\Gamma(s_1,s_2).
 \end{align*}
 Now, we fix $\epsilon > 0$ so that $\epsilon L - 2m < 0$ and $\alpha > 0$ so that $\frac L\epsilon - \alpha K < 0$ and we define $- c \coloneqq \max\{\epsilon L - 2m, \frac 1\alpha\left(\frac L\epsilon - \alpha K\right)\} < 0$; therefore we gain that:
  \begin{align*} \frac{\dd}{\dd t} &\int_{(\R^d)^2} \int_{S^2} (|x_1 - x_2|^2 + \alpha|s_1-s_2|^2)\,\dd \Pi_t(x_1,x_2,s_1,s_2)  \\
  &\leq -c\int_{(\R^d)^2} \int_{S^2} \left(|x_1 - x_2|^2 + \alpha|s_1 - s_2|^2\right)\dd\Pi_t(x_1, x_2, s_1, s_2) + \alpha KW_2^2(\mu_1, \mu_2).
  \end{align*}
  We can optimize the constant $c = \min\{2m - \epsilon L, K - \frac L{\alpha\epsilon}\}$ with $0 < \epsilon < \frac {2m}L$ and $\alpha > \frac L{\epsilon K}$, by choosing $\epsilon$ close enough to zero and $\alpha = \frac{n L}{\epsilon K}$ for $n > 1$ big enough. 
  We notice that $\min\{1, \alpha\}W_2^2(\rho_1(t), \rho_2(t)) \leq \int_{(\R^d)^2} \int_{S^2} (|x_1 - x_2|^2 + \alpha|s_1-s_2|^2)\,\dd \Pi_t(x_1,x_2,s_1,s_2)$ and $\int_{(\R^d)^2} \int_{S^2} (|x_1 - x_2|^2 + \alpha|s_1-s_2|^2)\,\dd \Pi_0(x_1,x_2,s_1,s_2) \leq \max\{1, \alpha\}\int_{(\R^d)^2} \int_{S^2} (|x_1 - x_2|^2 + |s_1-s_2|^2)\,\dd \Pi_0(x_1,x_2,s_1,s_2)$; hence, by applying again the Grönwall lemma, we get inequality \eqref{imp_ineq}. In particular, if in \eqref{imp_ineq} $\rho_{0,1}$ and $\rho_{0,2}$ are two stationary solutions of \eqref{fund_eq_meas} and $\mu_1 = \mu_2$, then we have that $\rho_{0,1} = \rho_{0,2}$; that is, if $f$ is strongly convex, the stationary state is unique. If $\mu_1 = \mu_2$ in \eqref{imp_ineq} and $\rho_0$ is any initial datum of $\rho$, we have \eqref{conv_stat}. One can easily verify by substitution in \eqref{fund_eq_meas} that this stationary state is equal to $\delta_{x^*}\otimes \mu$, if $x^*$ is the unique minimizer of $f$ for every $s \in S$. 
  \end{proof}
%$$ \int_{(\R^d)^2} \int_{S^2} |s_1 - s_2|^2\,\dd\Pi(t,x,s) =  \left(\int_{(\R^d)^2} \int_{S^2}|s_1 - s_2|^2\,\dd(\Pi_0 - \Gamma)(x,s)\right)e^{-Kt} + \int_{(\R^d)^2} \int_{S^2}|s_1 - s_2|^2\,\dd\Gamma(s)$$
%%%%%%%%%%%%%%%%%%%%%%%%%%%%%%%%%%%%%%%%%%%%%
\section{Stationary solutions} \label{exis_ss}

\subsection{Existence of Stationary Solutions} 
In this section, we discuss the existence of a stationary state in the general setting for the equation:
$$\frac{\partial \rho}{\partial t} = \nabla \cdot(\rho \nabla f(x, s)) + K(\bar \rho \otimes \mu - \rho) ,$$ 
with $\rho_0 \in \CMcal{P}_2(\R^d \times S)$, which solves the equation:
\begin{equation}\label{init_eq_meas} \nabla \cdot(\rho \nabla g) - \rho = -\bar \rho \otimes \mu, \end{equation}
or, in case $\rho$ is a probability density function, the equivalent equation for functions:
\begin{equation}\label{init_eq} \nabla \cdot(\rho \nabla g) - \rho = -\int_S \rho(x, s')\dd \mu(s'), \end{equation}
where $g \coloneqq f/K$. In order to obtain the existence of stationary solutions (and prevent loss of mass at $|x|=\infty$) we need to introduce some confinement assumptions.
Hence, throughout this section we will assume that there exist $c > 0$ and $\bar R > 0$
such that
$$\nabla f \cdot x \geq c|x|^2, $$ for every $|x| > \bar R$. This will provide us a uniform bound on the second moment, but one could obtain results under weaker conditions as well. The confinement assumption is satisfied in particular by the strong convexity hypothesis of theorem \ref{thm_conv} with $y = 0$. The strategy is approximating the problem on a bounded domain and then trying to pass to the limit on the whole space. 

%%%%%%%%%%%%%%%%%%%%%%%%%%%%%%%%%%%%%%%%%%%%
\subsubsection{Approximate problem on finite domains}
For a fixed $\epsilon > 0$, we consider the perturbed equation:
\begin{equation} \label{eq1}
\nabla \cdot(\rho \nabla g) + \epsilon\Delta\rho - \rho = -\int_S \rho(x, s')\dd \mu(s')
\end{equation}
on a bounded domain with no-flux boundary conditions. 
Since $v \coloneqq e^{-g/\epsilon}$ solves $\nabla \cdot(\epsilon \nabla v + v\nabla g) = 0$, we get the scaling $\rho \coloneqq e^{-g/\epsilon\,}w$. Then, $\rho \nabla g + \epsilon\nabla \rho = e^{-g/\epsilon\,}w\nabla g + \epsilon(-e^{-g/\epsilon\,}(\frac 1\epsilon)w\nabla g + e^{-g/\epsilon\,}\nabla w) = \epsilon e^{-g/\epsilon\,}\nabla w,$ that is, we have the equivalence $\nabla \cdot(\rho\nabla g + \epsilon \nabla \rho) = \nabla \cdot(\epsilon e^{-g/\epsilon\,}\nabla w)$, with which equation \eqref{eq1} reads:
\begin{equation}\label{eq2} - \nabla \cdot(\epsilon e^{-g/\epsilon\,}\nabla w) + e^{-g/\epsilon\,}w = \int_S e^{-g/\epsilon\,}w(x, s')\dd \mu(s'). \end{equation}
We fix $R > 0$ and, at first, we find a solution on the ball $B_R$ with center in zero and radius equal to $R$. For this purpose, we consider test functions $\phi \in L^2(S, H^1(B_R); \mu)$ and we write equation \eqref{eq2} in weak form as a Neumann problem with zero boundary condition:
$$\int_S\int_{B_R}\epsilon  e^{-g/\epsilon\,}\nabla w\cdot \nabla \phi +  e^{-g/\epsilon\,}w\phi\,\dd x\dd \mu(s) = \int_S\int_{B_R}\left(\int_S  e^{-g/\epsilon\,}w(x,s')\dd \mu(s')\right)\phi\dd x \dd \mu(s), $$
$\forall \phi \in  L^2(S, H^1(B_R); \mu)$. Note that the homogeneous Neumann boundary condition for $w$ corresponds to a no-flux boundary condition for the original variable $\rho$, which we want to use in order to maintain the mean-value, i.e. to obtain a probability density. This weak formulation can be interpreted as an eigenvalue problem for the function $w$: given a function $h \in L^2(S, L^2(B_R); \mu)$ we associate the function $A(h) \in L^2(S, H^1(B_R); \mu)$ solution of the variational problem:
\begin{equation}\label{eq3} \int_S\int_{B_R}\epsilon  e^{-g/\epsilon\,}\nabla A(h)\cdot \nabla \phi +  e^{-g/\epsilon\,}A(h)\phi\,\dd x\dd \mu(s) = \int_S\int_{B_R}\left(\int_S  e^{-g/\epsilon\,}h(x,s')\dd \mu(s')\right)\phi\dd x \dd \mu(s), \end{equation}
$\forall \phi \in  L^2(S, H^1(B_R); \mu)$. The linear operator $A: L^2(S, L^2(B_R); \mu) \to L^2(S, L^2(B_R); \mu)$ is well-defined and continuous thanks to the Lax-Milgram theorem: indeed the bilinear form is continuous and coercive, because the function $g$ is continuous on the whole space and $B_R \times S$ is bounded, so there exist $a, b > 0$, such that $a \leq e^{-g/\epsilon} \leq b$. So $\forall h \in  L^2(S, L^2(B_R); \mu), \exists! A(h) \in L^2(S, H^1(B_R); \mu) \subseteq L^2(S, L^2(B_R); \mu)$ solution of \eqref{eq3}. A solution of \eqref{eq2} is an eigenvector of the operator $A$ associated to the eigenvalue $1$. The operator $A$ is positive (i.e. if $h \geq 0,\, Ah \geq 0$ a.e.) and the proof of this is likewise to the one of the maximum principle. Moreover $A$ is a compact operator: let us take a bounded sequence $(h_n)_n \subseteq  L^2(S, L^2(B_R); \mu)$ and let $(w_n)_n$ be the images of $(h_n)_n$ through $A$, that is $w_n \coloneqq A(h_n)$. Thanks to this fact, we get the following easy inequality:
\begin{equation}\label{ineq1} \| w_n\|_{ L^2(S, H^1(B_R); \mu)} \leq C\left\|\int_S h_n(\cdot, s) \dd \mu(s)\right\|_{L^2(S, L^2(B_R); \mu)} \leq C\|h_n\|_{ L^2(S, L^2(B_R); \mu)} \leq C,\end{equation}
where $C$ is independent of $n$. Here, $C$ can indicate different constants. Thanks to the elliptic regularity of problem \eqref{eq3} and Nirenberg's method (cf. theorem $1$ of section $6.3$ in \cite{EV10}), we get also an estimate on the second derivatives which makes $w_n$ a $H^2$-functions in the variable $x$:
\begin{equation}\label{ineq2} \|D^2w_n\|_{ L^2(S, L^2(B_R); \mu)} \leq C'\left(\|h_n\|_{ L^2(S, L^2(B_R); \mu)} + \|w_n\|_{ L^2(S, H^1(B_R); \mu)}\right) \leq C', \end{equation}
with the same caveat as before. Another crucial inequality is concerning the translation operator in a generic Sobolev space, that is:
\begin{equation}\label{ineq3} \|\tau_hu - u\|_{L^2(\Omega)} \leq |h|\|\nabla u\|_{L^2(\Omega)},\,\,\forall u \in H^1(\Omega), \end{equation}
where $\tau_h$ is the translation of $u$ in the direction of the vector $h$. We fix $K \subset\subset B_R \times S$, $h\in \R^d$ and $t \in S$ such that $|(h, t)| < dist(K, \partial(B_R \times S))$. Whenever it makes sense, $w_n(\cdot - h, \cdot - t)$ is solution of the equation:
\begin{equation*}-\nabla \cdot(\epsilon e^{-g(x-h, s-t)/\epsilon\,}\nabla w_n(x-h, s-t)) +  e^{-g(x - h, s - t)/\epsilon\,}w_n(x - h, s-t) = \int_S  e^{-g(x - h, s')/\epsilon\,}h_n(x - h, s')\dd \mu(s'), \end{equation*}
for adequates $(x, s) \in K$. We approximate the integration on $K$ via cutoff functions $\psi_{K_m}$, that is regularized versions of the characteristic function of the subset $K$, by choosing $\phi = w_n\psi_{K_m}$. Hence $\nabla(w_n\psi_{K_m}) = \nabla(w_n)\psi_{K_m} + w_n\nabla\psi_{K_m}$, and we can pass to the limit exploiting the uniform integrability of $\psi_{K_m}$ and the Lebesgue theorem. We can control the condition on the translations exploiting \eqref{eq3}:
\begin{align*}&\int_K |w_n(x, s) - w_n(x - h, s-t)|^2\dd x \dd \mu(s) = \\ &\int_K w_n^2(x, s) + w_n^2(x - h, s-t) - 2w_n(x, s)w_n(x - h, s - t)\dd x\dd \mu(s) \leq \\ 
& C\int_K\Big( -\left|\nabla w_n\right|^2 + \left(\int_S h_n(x, s')\dd \mu(s')\right)w_n - \left|\nabla w_n(x - h, s- t)\right|^2 + \left(\int_S h_n(x - h, s')\dd \mu(s')\right)\cdot \\ &\cdot w_n(x - h, s - t) + 2\nabla w_n\cdot\nabla w_n(x - h, s - t) -  \left(\int_S h_n(x - h, s')\dd \mu(s')\right)w_n\, - \\ &-\left(\int_S h_n(x, s')\dd \mu(s')\right)w_n(x - h, s - t)\Big)\dd x \dd \mu(s) \leq C\int_K \Big(2|\nabla w_n||\nabla w_n - \nabla w_n(x - h, s-t)| + \\ &\left|\int_S h_n(x - h, s')\dd \mu(s')\right||w_n - w_n(x - h, s - t)| + \left|\int_S h_n(x, s')\dd \mu(s')\right||w_n - w_n(x - h, s - t)|\Big)\dd x\dd \mu(s) 
\\[2mm] &\leq C\big(2\|\nabla w_n\|_{L^2(K)}\|\tau_{(h, t)}\nabla w_n - \nabla w_n\|_{L^2(K)} + 2\|h_n\|_{L^2(K)}\| \tau_{(h, t)}\nabla w_n - \nabla w_n\|_{L^2(K)}\big) \leq C'|(h, t)|,
\end{align*}
using \eqref{ineq1}, \eqref{ineq2}, \eqref{ineq3} and the Cauchy-Schwarz inequality, with $C$ constant depending on $g$ and $\epsilon$. After this calculation, we have the translation equicontinuity property and so we are allowed to apply the Fréchet-Kolmogorov theorem (theorem $4.26$ in \cite{BR11}) on the sequence $(w_n)_n$ in $L^2(S, L^2(B_R); \mu)$, which provides us the compactness of the operator $A$. In this case, the equitightness hypothesis of the theorem is not required, because the set $B_R \times S$ is bounded. \\
One can notice that the original problem for $\rho$ \eqref{eq1} is equivalent to solve the eigenvalue problem, not for the operator $A$, but for the operator $AB$, where $B: L^2(S, L^2(B_R); \mu) \to L^2(S, L^2(B_R); \mu)$ is such that $B(\rho) \coloneqq e^{\,g/\epsilon\,}\rho = w$. With this definition, $B$ is a linear bounded operator and therefore $L \coloneqq AB$ is still a compact (and positive) operator. We can apply the Krein-Rutman theorem (cf. theorem $1.1$ in \cite{DU06}) on the operator $L$, getting that the spectral radius of $L$, $r(L)$, is an eigenvalue of $L$ with a nonnegative eigenvector $\rho \in L^2(S, H^1(B_R); \mu) $, since the range of $A$ is in $H^1(B_R)$, which solves the weak formulation:
\begin{equation*}r(L)\int_S\int_{B_R}\rho \nabla g \cdot \nabla \phi + \epsilon \nabla\rho \cdot \nabla\phi + \rho\phi\, \dd x\dd \mu(s) = \int_S\int_{B_R}\left(\int_S \rho(x, s')\,\dd \mu(s')\right)\phi\,\dd x\dd\mu(s), \end{equation*}
$\forall \phi \in  L^2(S, H^1(B_R); \mu)$. To have an estimate on the value of $r(L)$, we use the fact that the adjoint operator $L^*$ of $L$ has the same real eigenvalues. Applying the adjoint operator of L is equivalent to find, given a function $\omega$, a solution $\phi$ such that:
$$(\epsilon\Delta\phi - \nabla\phi\cdot\nabla g - \phi) = - \int_S\omega(x, s')\,\dd\mu(s'),$$
and, if $\omega \equiv \phi \equiv constant > 0$, we have that $\lambda = 1$ is an eigenvalue of $L^*$, hence of $L$. If, \emph{ad absurdum}, $r(L) > 1$ and $\rho$ were its nonnegative eigenvector, with $\phi \equiv 1$:
\begin{equation}\begin{gathered}\label{eq4} \int_S\int_{B_R}\rho(x, s)\,\dd x\dd\mu(s) = \int_S\int_{B_R}\phi\rho(x, s)\,\dd x\dd\mu(s) = (\phi, \rho)_{L^2(S, L^2(B_R); \mu)} = \\ = (L^*\phi, \rho)_{L^2(S, L^2(B_R); \mu)} = (\phi, L\rho)_{L^2(S, L^2(B_R); \mu)}. \end{gathered} \end{equation}
Moreover:
\begin{equation}\begin{gathered}\label{eq5} r(L)\int_S\int_{B_R}\rho(x, s)\,\dd x\dd\mu(s) = \int_S\int_{B_R}\phi(r(L)\rho)(x, s)\,\dd x\dd\mu(s) = (\phi, r(L)\rho)_{L^2(S, L^2(B_R); \mu)} \\ =
(\phi, L\rho)_{L^2(S, L^2(B_R); \mu)}, \end{gathered} \end{equation}
and putting together \eqref{eq4}, \eqref{eq5}, we would have that, if $r(L) > 1$, $\rho \equiv 0$. So $r(L) = 1$ with $\rho \geq 0$ and the weak formulation reads:
\begin{equation}\int_S\int_{B_R}\rho \nabla g \cdot \nabla \phi + \epsilon \nabla\rho \cdot \nabla\phi + \rho\phi\, \dd x\dd \mu(s) = \int_S\int_{B_R}\left(\int_S \rho(x, s')\,\dd \mu(s')\right)\phi\,\dd x\dd\mu(s), \label{eq6} \end{equation}
with $\rho \in L^2(S, H^1(B_R); \mu)$, $\forall \phi \in  L^2(S, H^1(B_R); \mu)$. \\
%%%%%%%%%%%%%%%%%%%%%%%%%%%%%%%%%%%%%%%%%%%%%%%%%
\subsubsection{Vanishing diffusivity limit and problem on the whole space}
The next step is eliminating the perturbation and with $\rho_\epsilon$ we shall indicate the solution of \eqref{eq6} on a ball $B_R$, with radius $R$ fixed, corresponding to the fixed value of $\epsilon$. We choose $R > 0$ such that $R > \bar R$, where the value $\bar R > 0$ is given by the hypothesis considered at the beginning of this section. Without any change of notation, we rescale $\rho_\epsilon$, that is, we consider $\|\rho_\epsilon\|_{L^1(S, L^1(B_R); \mu)} = 1$ and now $\rho_\epsilon$ is a probability density function on $B_R \times S$. We would like to estimate the term $\epsilon \nabla \rho_\epsilon \cdot \nabla\phi$ when $\epsilon \rightarrow 0$ exploiting weak convergence properties, for whom boundedness of the second moments is enough. To this purpose, we use equation \eqref{eq6} with test function $\phi(x, s) = x_i^2$ getting
$\int_S\int_{B_R}\frac{\partial g}{\partial x_i}x_i\rho_\epsilon\,\dd x\dd\mu(s) + \epsilon\int_S\int_{B_R}x_i\frac{\partial\rho_\epsilon}{\partial x_i}\,\dd x\dd\mu(s) = 0$. On $S$ there is no problem of boundedness of moments, since $S$ is bounded. By summing on $i$ these equations, we gain:
$$\int_S\int_{B_R}(\nabla g \cdot x)\rho_\epsilon\,\dd x\dd\mu(s) + \epsilon\int_S\int_{B_R}\nabla\rho_\epsilon \cdot x\,\dd x\dd\mu(s) = 0.$$
The second term is below-bounded, indeed by integration by parts:
$$\int_S\int_{B_R}\frac{\partial \rho_\epsilon}{\partial x_i}x_i\,\dd x\dd\mu(s) = -\int_S\int_{B_R}\rho_\epsilon\,\dd x\dd\mu(s) + \int_S\int_{\partial B_R}\rho_\epsilon x_i \,\underline n_i\,\dd x\dd\mu(s),$$
and we sum harnessing the boundedness in $L^1$ and the fact that, on $\partial B_R$, $x \cdot \underline n = |x|^2/R = R$:
$$\int_S\int_{B_R}\nabla\rho_\epsilon \cdot x\,\dd x\dd\mu(s) = -d + \int_S\int_{\partial B_R}\rho_\epsilon x\cdot \underline n\,\dd x\dd\mu(s)= - d + R\int_S\int_{\partial B_R}\rho_\epsilon\,\dd x\dd\mu(s) \geq - d.$$
\begin{lemma}\label{mom_lem}Under the given hypothesis, the second moments of the set $(\rho_\epsilon)_{0 < \epsilon \leq 1}$ are uniformly bounded and therefore there exists a sequence weakly convergent to a probability measure $\rho_R$ on $B_R \times S$ when $\epsilon \to 0$. Moreover, $m_2(\rho_R) \coloneqq \int_S\int_{B_n}|x|^2\dd\rho_R(x, s) \leq C$, with $C > 0$ independent of $R$. \end{lemma}
\begin{proof} If $m_2(\rho_\epsilon) \coloneqq \int_S\int_{B_n}|x|^2\rho_\epsilon(x, s)\,\dd x\dd\mu(s)$ is the second moment of $\rho_\epsilon$, we remember the assumption $\nabla g \cdot x \geq c|x|^2$ when $|x| > \bar R$ and hence, since $R > \bar R$:
\begin{align*} -cm_2(\rho_\epsilon) &= \int_S\int_{B_R}\rho_\epsilon(x \cdot \nabla g - c|x|^2)\,\dd x\dd \mu(s) + \epsilon\int_S\int_{B_R}\nabla\rho_\epsilon \cdot x\,\dd x\dd\mu(s) \geq \\ &\geq \int_S\int_{B_{\overline R}}\rho_\epsilon(x \cdot \nabla g - c|x|^2)\,\dd x\dd \mu(s) - d \,\geq\, \\ &\geq -(\bar RG_{\bar R} + c{\overline R}^2)\int_S\int_{B_{\bar R}}\rho_\epsilon\,\dd x\dd\mu(s) -d \geq -(\bar RG_{\bar R} + c{\bar R}^2 + d),
\end{align*}
since on $B_{\bar R}$, $|x \cdot \nabla g| \leq \bar RG_{\overline R}$, where $G_{\bar R} \coloneqq \sup_{(x, s) \in B_{\bar R} \times S}|\nabla g|$. Then, $m_2(\rho_\epsilon) \leq  \frac{\bar RG_{\bar R} + c{\bar R}^2 + d}c$, so the set $(\rho_\epsilon)_{0 < \epsilon \leq 1}$ is tight. Therefore by Prokhorov's theorem, up to a subsequence, $\exists \rho_R \in \CMcal{P}(B_R \times S)$ such that $\rho_\epsilon \overset{w}{\rightharpoonup} \rho_R$ when $\epsilon \to 0$. Moreover, since the domain is bounded and $x^2$ is bounded and continuous, for the second moment of $\rho_n$ the same estimate holds.
\end{proof}
In \eqref{eq6}, we set $\phi \in C_b(S, \D(B_R))$, where $C_b(S, \D(B_R))$ is the space of bounded continuous functions on $S$ and $C^\infty$ with compact support on $B_R$ and through regularity of $\phi$, we obtain:
\begin{equation}\label{eq_diff} \int_S\int_{B_R}\rho_\epsilon \nabla g \cdot \nabla \phi - \epsilon\, \rho_\epsilon \cdot \Delta\phi + \rho_\epsilon\phi\, \dd x\dd \mu(s) = \int_S\int_{B_R}\left(\int_S \rho_\epsilon(x, s')\,\dd \mu(s')\right)\phi\,\dd x\dd\mu(s), \end{equation}
but when $\epsilon \to 0$, $\rho_\epsilon \overset{w}{\rightharpoonup} \rho_R$
and then for every $\phi \in C_b(S, \D(B_R))$, the sequence $\int_S\int_{B_R} \rho_\epsilon \cdot \Delta \phi\,\dd x\dd\mu(s)$ is bounded, i.e. $\rho_R$ solves the relation:
\begin{equation}\int_S\int_{B_R}(\nabla g \cdot \nabla \phi + \phi)\, \dd \rho_R(x, s) = \int_S\int_{B_R}\phi\,\dd \bar \rho_R(x)\dd\mu(s), \end{equation}
where $\bar \rho_R$ is the marginal probability measure of $\rho_R$ on $B_R$, for every $\phi \in C_b(S, \D(B_R))$. Moreover, for every $R > \bar R$, $m_2(\rho_R) \leq  \frac{\bar RG_{\bar R} + c{\bar R}^2 + d}c$. \\
The last unsolved issue is the extension of $\rho_R$ to the whole space $\R^d$ and we define the sequence $(\rho_n)_{n \in \N}$, where $\rho_n$ is a solution obtained in the previous step with $R = n$ on the ball $B_n$ with radius $n$. We extend them trivially to a measure on the whole space $\tilde \rho_n$, such that for every set $A \in \mathcal{B}(\R^d \times S)$, $\tilde \rho_n(A) \coloneqq \rho_n(A \cap (B_n \times S))$. Since we have a uniform bound on the second moments of the sequence $(\tilde \rho_n)_n$, this sequence is tight and, at least for a subsequence, we get that $\exists \rho \in \CMcal{P}_2(\R^d \times S)$ such that $\tilde \rho_n \overset{w}{\rightharpoonup} \rho$. We can summarize these results in a unique final theorem, exploiting the weak convergence and remembering that $g = f/K$.
\begin{thm}[Existence of stationary states]\label{ex_ss} There exists a probability measure $\rho \in \CMcal P_2(\R^d \times S)$ distributional solution of the equation \eqref{init_eq_meas}. More precisely, the measure $\rho$ satisfies:
\begin{equation}\label{fin_eq} \int_S\int_{\R^d}(\nabla f \cdot \nabla \phi + K\phi)\,\dd \rho(x, s) = K\int_S\int_{\R^d}\phi\, \dd\bar \rho(x)\dd\mu(s), \end{equation}
for every $\phi \in C_b(S, \D(\R^d))$, where $C_b(S, \D(\R^d))$ is the space of bounded continuous functions on $S$ and $C^\infty$ with compact support on $\R^d$ and $\bar \rho$ is the marginal probability measure of $\rho$ on $\R^d$. \end{thm}
As already highlighted in section \ref{scms}, when $f(\cdot,s)$ has the same minimizer $x^*$ for all $s$, then the product measure $\rho = \delta_{x^*} \otimes \mu$, where $\delta_{x^*}$ is the Dirac measure centered in the point $x^*$, is a solution of \eqref{fin_eq}, since $f$ is a regular function when $s$ is fixed, and it is unique. This does not hold in general, because this reasoning can be iterated in case $f$ has $N$ minimizers, $x_1^*, \dots, x_N^*$, for all $s$, then $\delta_{x_i^*} \otimes \mu$, for every $i = 1, \dots, N$ are solutions of \eqref{fin_eq} and hence every convex combination of $\delta_{x_i^*} \otimes \mu$. 
Nevertheless, we can write the following coupling equation:
\begin{equation} \label{coup_ss}
    \partial_t \Pi = \nabla_x\cdot(\Pi\nabla_s f(x, s_1)) + \nabla_y\cdot(\Pi\nabla_y f(y, s_2)) + K\int_{S^2}(\Gamma(s_1, s_2)\Pi(x, y, s_1', s_2') - \Pi)\dd s_1'\dd s_2'
\end{equation}
with $\Gamma = \mu\delta_{s_1 - s_2}$, which is a coupling between two stationary solutions $\rho_{\infty, 1}$, $\rho_{\infty, 2}$ of \eqref{fund_eq_meas} if the initial datum $\Pi_0$ is a coupling between such stationary states. Such $\Pi$ is well-defined thanks to theorem \ref{ex_thm}. We multiply equation \eqref{coup_ss} by $|s_1 - s_2|^2$ and integrate, obtaining that 
\begin{align*}
W_2(\pi_S\#\rho_{\infty, 1}, \pi_S\#\rho_{\infty, 2})^2 \leq \int_{\R^{2d}\times S^2}|s_1 - s_2|^2&\dd \Pi_t(x, y, s_1, s_2) = \\ &\left(\int_{\R^{2d}\times S^2}|s_1 - s_2|^2\dd \Pi_0(x, y, s_1, s_2)\right)e^{-Kt},
\end{align*}
where $\pi_S$ is the projection on $S$. But this is impossible, unless $\int_{\R^{2d}\times S^2}|s_1 - s_2|^2\dd \Pi_0(x, y, s_1, s_2) = 0$ and then $W_2(\pi_S\#\rho_{\infty, 1}, \pi_S\#\rho_{\infty, 2}) = 0$. This means that two different stationary states have the same marginal on $S$.
\\
We underline that one can obtain a result of existence of stationary states also for the equation with diffusion, that is equation \eqref{eq_diff} with $\epsilon > 0$, exploiting the same argument and lemma \ref{mom_lem}.
%%%%%%%%%%%%%%%%%%%%%%%%%%%%%%%%%%%%%%%%%%%%%%%%%%%%%%%%%%%%%%%%%%%%%%%%%%%%%%%%%%%%%%%%%%%%%%%

\subsection{Support of stationary solutions}

In this subsection, we would like to discuss the support of the stationary solutions of \eqref{fund_eq_meas}. Indeed,
Dirac deltas are not the only possible stationary states: for instance there are "diffused" stationary states given by $d = 1$, $S = \{1, 2\}$, $\mu = Bern(0.5)$ and $f(x, s) = \frac 12s|x - s|^2$. In this case a stationary solution is $\rho(x, 2) = c_2(2 - x)^{\frac{K - 4}4}(x - 1)^{\frac K2}\mathbbm{1}\{1 < x < 2\}$, $\rho(x, 1) = c_1\frac{2 - x}{x - 1}\rho(x, 2)$ 
where $c_1, c_2$ are normalization constants.
\begin{lemma} \label{Lemma:stationarysupport}
Let $\psi = \psi(x), \varphi = \varphi(x) \geq 0$ be such that $\nabla f(x, s) \cdot \nabla \psi(x) \geq \varphi(x)$. Then, if $\bar \rho$ is the marginal of $\rho$ on $\R^d$, $Supp(\bar \rho) \subseteq \overline{\{\varphi = 0\}}$.
\end{lemma}
\begin{figure} 
    \centering
    \begin{subfigure}[b]{0.49\textwidth}
        \centering
        \includegraphics[width=\textwidth]{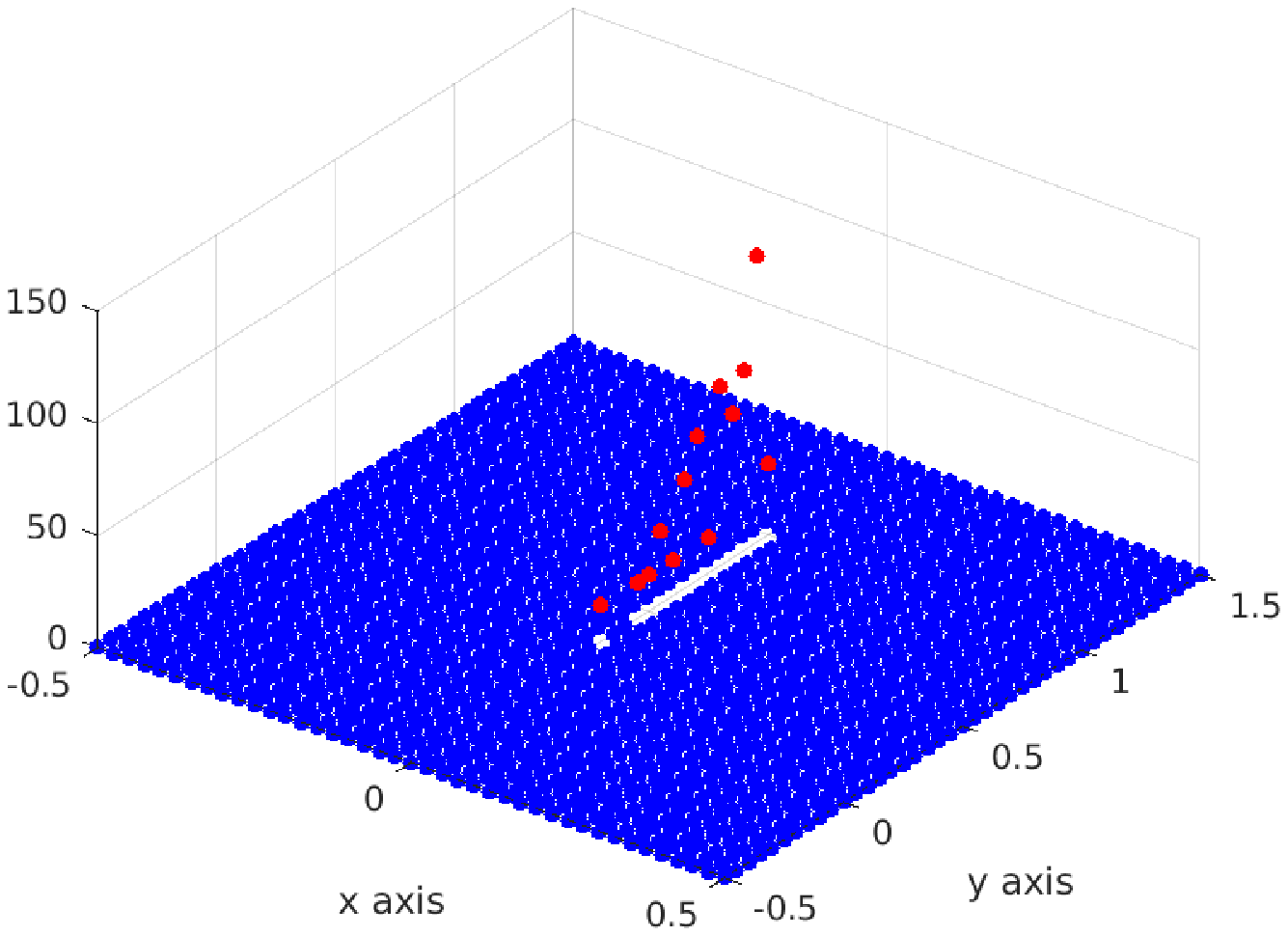}
    \end{subfigure}
    \begin{subfigure}[b]{0.49\textwidth}
        \centering
        \includegraphics[width=\textwidth]{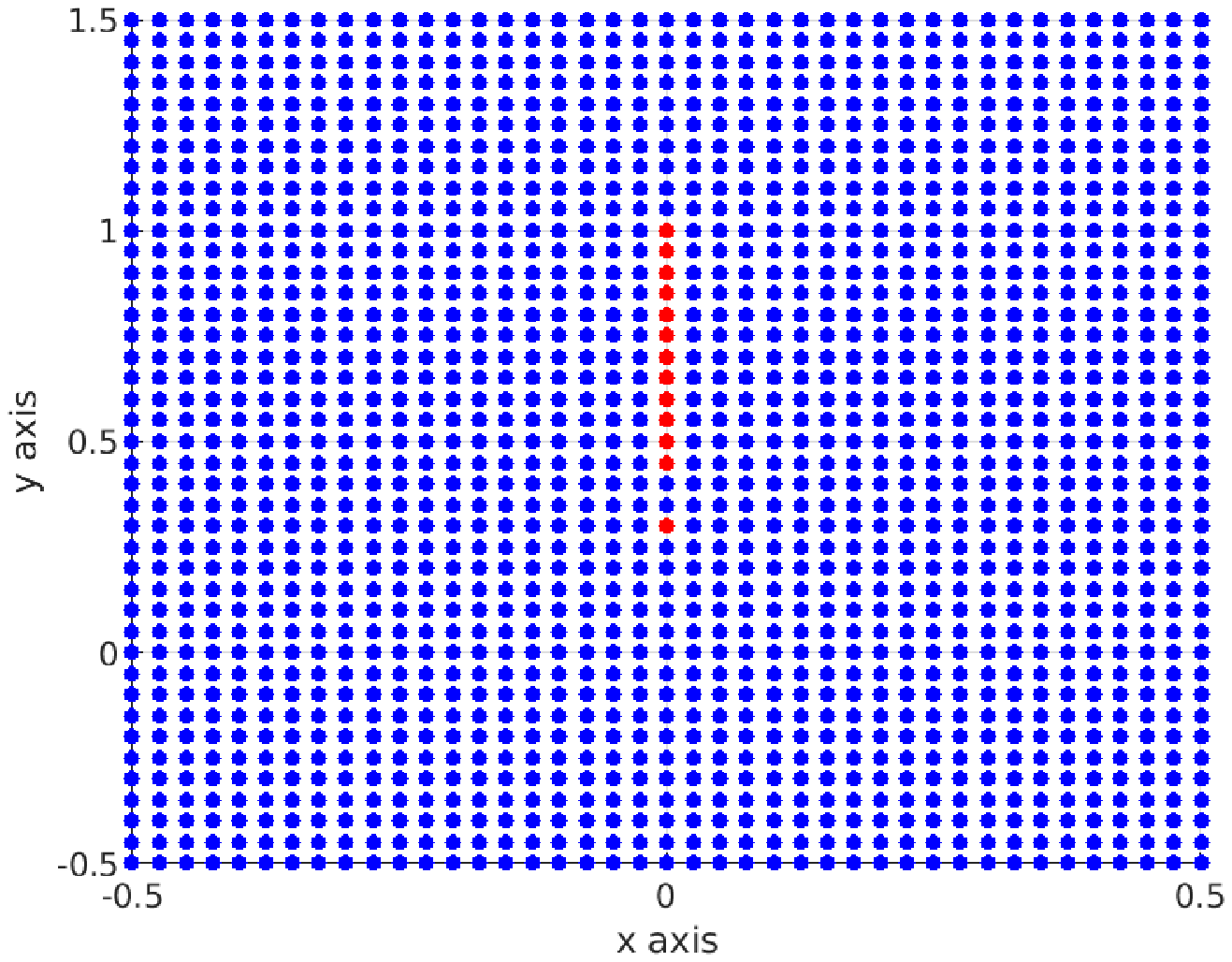}
    \end{subfigure}
    \caption{Stationary state of example \ref{ex_eig} with $v = (0, 1)$ and $p = 0.7$. The red dots represent the support of this state.}
    \label{fig_eig_ss}
    \centering
    \begin{subfigure}[b]{0.49\textwidth}
        \centering
        \includegraphics[width=\textwidth]{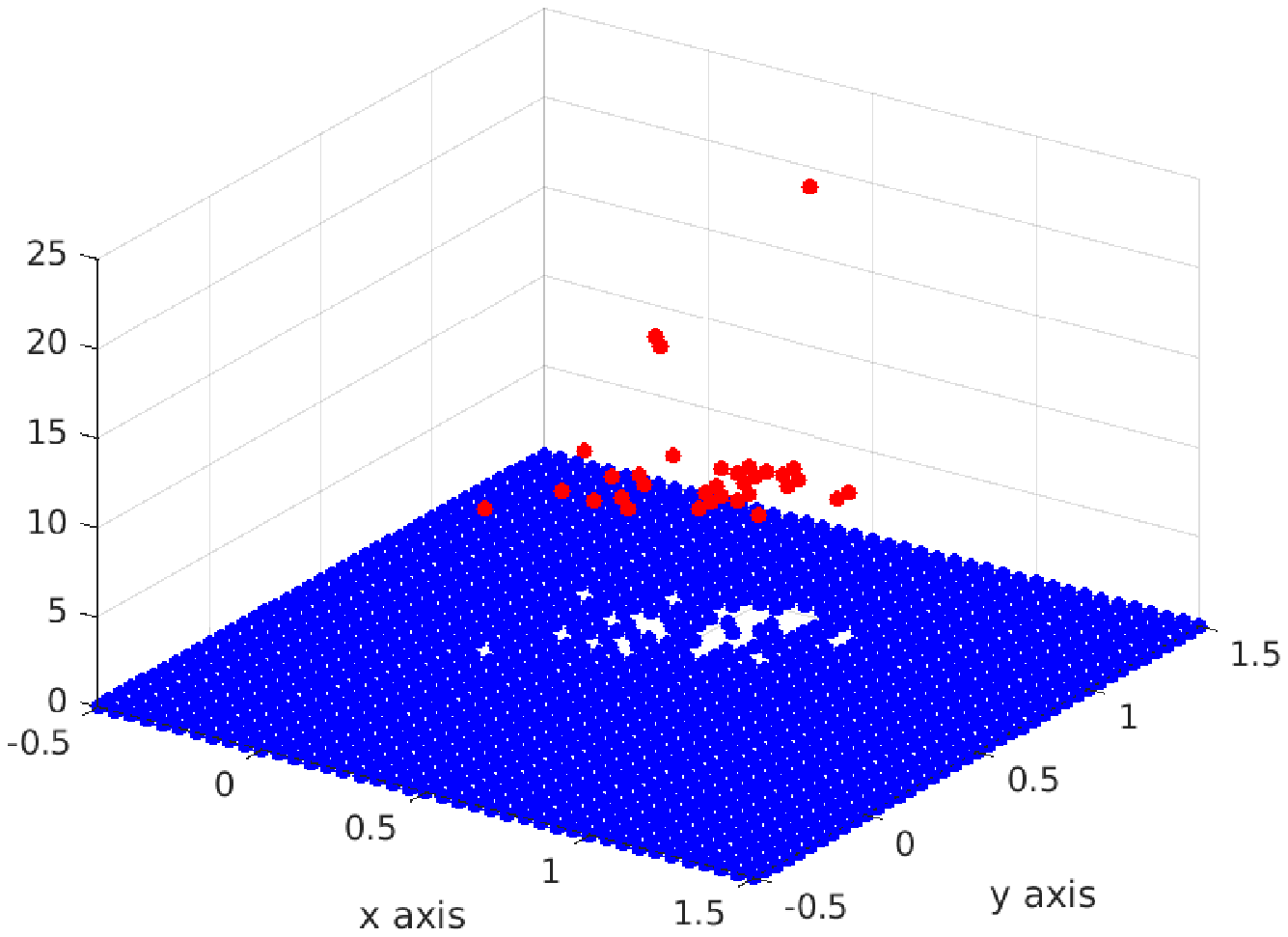}
    \end{subfigure}
    \begin{subfigure}[b]{0.49\textwidth}
        \centering
        \includegraphics[width=\textwidth]{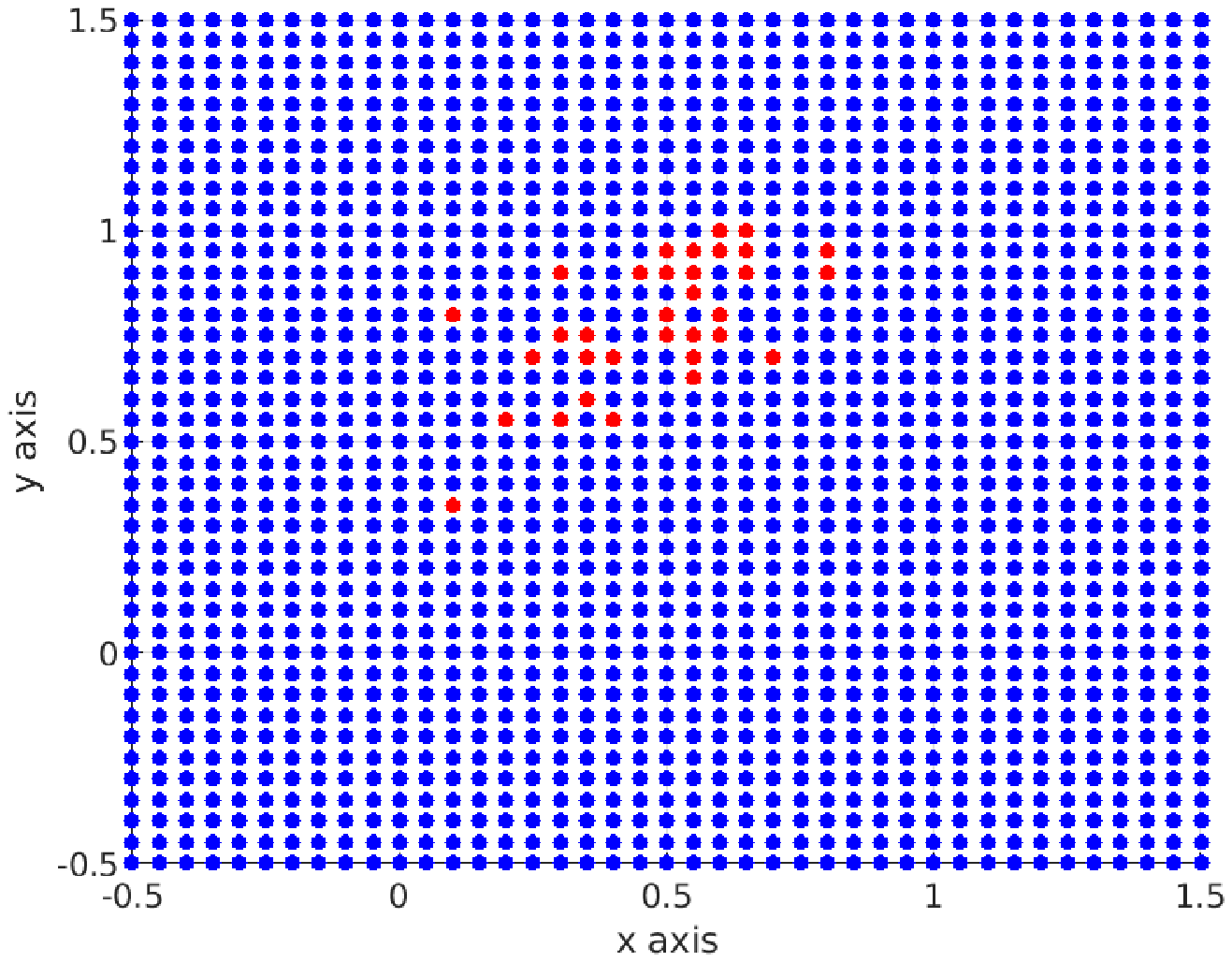}
    \end{subfigure}
    \caption{Stationary state of example \ref{ex_eig} with $v = (1, 1)$ and $p = 0.7$. The red dots represent the support of this state.}
    \label{fig_neig_ss}
\end{figure}
\begin{proof}If we integrate \eqref{init_eq} on $S$, we obtain $0 = \nabla \cdot \left(\int_S \rho \nabla f \,\dd \mu(s)\right)$. Then, we take $\psi, \varphi$ as in the hypothesis, gaining:
\begin{equation*}0 = \int_{\R^d}\int_S \left(\nabla f \cdot \nabla \psi \right)\,\dd \rho(x, s) \geq  \int_{\R^d}\int_S \varphi(x) \,\dd \rho(x, s) =  \int_{\R^d} \varphi(x) \,\dd \bar{\rho}(x).
\end{equation*}
Since $\varphi \geq 0$, this gives us the thesis.
\end{proof} 
\begin{example}Under the assumptions, we have that $\nabla f(x, s) \cdot x \geq c|x|^2$ if $|x| > \bar R$ and therefore we choose in the previous lemma $\psi(x) = \frac{|x|^2}2$ if $|x| > \bar R$, $\psi(x) = \frac{{\bar R}^2}2$ elsewhere, and $\varphi(x) = c|x|^2$ if $|x| > \bar R$, $\varphi(x) = 0$ elsewhere. Then we have that $Supp(\bar \rho) \subseteq \overline{B_{\bar R}(0)}$, that is every stationary state is compactly supported.
\end{example}
\begin{example}
Let $f$ be strongly convex for every $s \in S$ with minimizer $x_s \in \R$ and $d = 1$, then we define $x_0 \coloneqq \inf_{s \in S} x_s$ and $x_1 \coloneqq \sup_{s \in S} x_s$. We choose $\psi$ and $\varphi$ such that:
\begin{equation*}
\psi(x) \coloneqq 
\begin{cases} 
      f(x, s_0) & x < x_0 \\
      f(x_0, s_0) &  x_0 \leq x \leq x_1 \\
      f(x, s_0) - f(x_1, s_0) + f(x_0, s_0) & x > x_1 
\end{cases}
\end{equation*}
\begin{equation*}
\varphi(x) \coloneqq 
\begin{cases}
0 & x \in [x_0, x_1] \\
\inf_{s \in S}(\partial_x f(x, s) \partial_x f(x, s_0)) & elsewhere
\end{cases}
\end{equation*}
where $s_0 \in S$ is fixed. $\varphi \geq 0$ thanks to the strong convexity and $\varphi(x) \neq 0$ if $x \notin [x_0, x_1]$. This way we get that $Supp(\bar \rho) \subseteq  [\inf_{s \in S} x_s, \sup_{s \in S} x_s]$, that is the support of the stationary state is in the convex hull of the minimizers in a $1$-dimensional setting.
\end{example}

\begin{example}
Let us consider the case  $d = 2$, $S = \{1, 2\}$, $\mu = Bern(p)$ for some $p \in (0,1)$. We choose 
$$  f(x, 1) = |x|^2, \qquad f(x, 2) = |x-v|^2  $$
with  $v = (1, 1)$. With Lemma \ref{Lemma:stationarysupport} and the choice $\psi(x) = 1/2$ if $|x - \frac 12 v| \leq \frac 1{\sqrt 2}$ and    $\psi(x) = |x-\frac 1 2 v|^2$ otherwise, we conclude immediately that the support of the stationary solution is contained in the ball with radius $\frac{1}{\sqrt{2}}$ around $\frac{1}2 v$. Now we can choose a test function with compact support such that $\phi = |x_1 - x_2|^2$ in this ball and hence obtain from the weak formulation
$$0 = \int_{\R^d}\int_S \left(\nabla f \cdot \nabla \phi \right)\,\dd \rho(x, s)
= 4 \int_{\R^d}\int_S \left(x_1-x_2 \right)^2\,\dd \rho(x, s).
$$
This further implies that the support of the stationary solution is contained in the line $x_1=x_2$, and by intersecting with the above ball this implies the support is contained in the line segment between $(0,0)$ and $(1,1)$. On this line segment the stationary solution can be computed explicitely as in the previous one-dimensional example. Hence, we obtain a unique stationary solution supported in the convex hull of the minimizers of the functionals $f(\cdot,s)$
 \end{example}

In general, one cannot conclude that the support of a stationary state is contained in the convex hull of the minimizers of $f$ as we see from the following modification of the previous example.

\begin{example} \label{ex_eig}
Let us consider the case  $d = 2$, $S = \{1, 2\}$, $\mu = Bern(p)$ for some $p \in (0,1)$. We choose 
$$  f(x, 1) = |x|^2, \qquad f(x, 2) = |A(x - v)|^2 , \quad A = \text{diag}\left(\frac 1{\sqrt{2}}, 1\right)$$
with  $v = (1, 1)$. Let us assume there is a stationary solution with support in the line segment between $(0,0)$ and $(1,1)$. Denote by $x_t = tv$ and by
$(\eta_1,\eta_2)$ the restrictions of the stationary measures to the support line $\{tv~|~0\leq t \leq 1\}$. Now we choose a test function such that $\phi(x) = x_1-x_2$ on the support line and we obtain from the weak formulation
$$ 0 = \int_{\R^d}\int_S \left(\nabla f \cdot \nabla \phi \right)\,\dd \rho(x, s)  $$
that
$$ 0 = (1 - p) \int_0^1 (1 - x_2) \,\dd\eta_2.$$
Thus, $\eta_2$ is concentrated in $t=1$, which implies that: 
$$ 0 =   \nabla \cdot (\nabla f(\cdot,2)  \rho(\cdot,2)) $$ and hence 
$\rho(\cdot,1) = \frac p{1 - p}\rho(\cdot,2)$, since $\rho(\cdot, 2) = \overline \rho(\cdot)\mu(2) = \rho(\cdot, S)\mu(2) = (\rho(\cdot, 1) + \rho(\cdot, 2))\mu(2)$. $\rho(\cdot,1)$ concentrated in $v$ is however not a solution of the stationary problem, because $\nabla f(v, 1) \neq 0$ in \eqref{fin_eq}.
 We numerically investigate with a Monte Carlo method the case $p=0.7$: in figure \ref{fig_eig_ss} we fix $v = (0,1)$ and in this case the stationary solution has support contained in the convex hull of the minimizers of $f$ $(0, 0)$ and $(0, 1)$. If we fix $v = (1, 1)$ as in figure \ref{fig_neig_ss}, the support of the stationary state in not contained in the convex hull of the minimizers $(0, 0)$ and $(1, 1)$ any longer.
\end{example}
%%%%%%%%%%%%%%%%%%%%%%%%%%%%%%%%%%%%%%%%%%%%%%%%%%%%%%

\section{Convergence analysis} \label{conv_an}
In this section, we would like to analyze the limit of solutions of equation \eqref{fund_eq_meas} if $K \to \infty$.
We define the function $F(x) \coloneqq \E_s[f(x,s)]$, with $F$ $m$-strongly convex with $x^*$ unique minimizer. Note that this assumption is weaker than the strong convexity of $f$ we considered in section \ref{scms}. Moreover, if we fix the initial datum $\rho_0$ appropriately, that is such that the marginal on $S$ is $\mu$, then $\mu$ is a stationary marginal measures for $\rho_t$ for every $t > 0$. This can be seen integrating \eqref{fund_eq_meas} against a function $\phi \in C_b(\R^d \times S)$ constant in $x$. We can define $\sigma^2 \coloneqq \sup_{x \in \R^d}\E_s\left[|\nabla f(x, s) - \nabla F(x)|^2\right]$, which can be considered as a variance term, since $\E_s[\nabla f(x, s)] = \nabla F(x)$. Thanks to the Lipschitz continuity of $\nabla f$, this variance term is bounded:
\begin{equation*}|\nabla F(x) - \nabla f(x,s)| \leq \int_S|\nabla f(x, \omega) - \nabla f(x, s)|\,\dd\mu(\omega) \leq L\sup_{s, \omega \in S}|\omega - s|. \end{equation*}
%%%%%%%%%%%%%%%%%%%%%%%%%%%%%%%%%%%%%%%%%%%%%%%%%%%%%%%%%%%%%%%%%%%%%%%%%%%%%%%%%%%%%%%%%%%%%%%%%%%%%%%%%%%%
\subsection{Grazing collision limit}  \label{graz_coll}
In this subsection, we would like to discuss the limit $K \to \infty$. For this purpose, we introduce the formal Hilbert expansion:
\begin{equation} \label{hilb_ex}
\rho = \sum_{n = 0}^{+\infty} \frac 1{K^n}\rho_n.
\end{equation}
The actual convergence of this series is not discussed. From this expansion, we obtain a hierarchy of equations: $\rho_0$ solves 
$\rho_0(x, s) =\bar{\rho_0} \otimes \mu $
and, if $n \geq 1$, 
$$\partial_t \rho_{n-1} = \nabla \cdot(\rho_{n-1}\nabla f) + \int_S(\rho_n(x, s') - \rho_n(x, s))\,\dd\mu(s').$$ Integrating on $S$ the equation of order $1$, we get that $\rho_0 =\eta \otimes \mu$, with $\eta$ solving the gradient flow equation :
\begin{equation} \label{eq:gradflow} \partial_t \eta = \nabla \cdot (\eta \nabla F(x))  \end{equation}  since $\eta$ is independent of $s$, thanks to the equation of order zero, and when $K \to \infty$, we expect $\rho$ converging to $\rho_0$. We can interpret $\frac 1K$ as the "learning rate" of the equation and we expect that when $K \to \infty$ the integral part of the equation is nullified and $\rho$ converges towards of the solution of the gradient flow. 

In order to make the analysis more rigorous, we write the following coupling equation for $\Pi$:
\begin{equation} \begin{aligned}\label{coup_eq2}
     \frac{\partial \Pi_t}{\partial t}  =& \, \nabla_x \cdot(\Pi_t \nabla_x f(x, s_1)) +  \nabla_y \cdot(\Pi_t \nabla_y F(y)) 
     + \\ & + K~ \int_S \int_S (\Gamma(s_1, s_2)\Pi_t(x, y, s_1', s_2') - \Pi_t(x, y, s_1, s_2))\dd \mu(s_1') \dd \mu(s_2'),
 \end{aligned} \end{equation}
 with initial datum $\Pi_0$ coupling of $\rho_0$ and $\eta_0 \otimes \mu$, $\Gamma = \mu \delta_{s_1 - s_2}$ probability measure on $S^2$ such that $\Gamma(diag(S^2)) = 1$ and $\eta = \eta(t, y)$ solution of the gradient flow $\partial_t \eta = \nabla_y \cdot (\eta(y)\nabla F(y))$ with initial datum $\eta_0$. Equation \eqref{coup_eq2} is also well--posed thanks to theorem \ref{ex_thm}. The marginals of $\Pi_t$ are  $\rho_t$ and $\eta_t \otimes \mu$ by uniqueness. Using these properties we can derive the following result:
 \begin{prop}Let $F$ be strongly convex with modulus $m$. Moreover, let $\rho_t$ be the unique solution of \eqref{fund_eq_meas} with initial value $\eta_0 \otimes \mu$ and $\eta_t$ the unique solution of the gradient flow \eqref{eq:gradflow} with initial value $\eta_0$. Then the following estimate holds
\begin{equation} \label{eq:grazingestimate1} W_2(\rho_t, \eta_t \otimes \mu) \leq W_2(\rho_0, \eta_0 \otimes \mu)e^{\frac{\max\{-m, -K\}}2t} + \sqrt{\frac{\sigma^2}m\frac{1 - e^{-mt}}m} .\end{equation}
 \end{prop}
 \begin{proof}
We employ  \eqref{coup_eq2} to obtain
 $$\frac \dd{\dd t}\int_{\R^{2d} \times S^2}|x - y|^2\dd\Pi_t(x,y,s_1,s_2) = -2\int_{\R^{2d} \times S^2}(\nabla f(x,s_1) - \nabla F(y))\cdot(x - y)\dd\Pi_t(x,y,s_1,s_2), $$
 and we can then exploit the strong convexity of $F$ and Young's inequality to obtain
\begin{equation}\begin{aligned}\label{var_term} (&\nabla f(x,s_1) - \nabla F(y))\cdot(x - y) = (\nabla f(x,s_1) - \nabla F(x) + \nabla F(x) - \nabla F(y))\cdot(x - y) \geq \\
&(\nabla f(x,s_1) - \nabla F(x))\cdot(x - y) + m|x - y|^2 \geq \frac m2|x - y|^2 - \frac 1{2m}|\nabla f(x,s_1) - \nabla F(x)|^2. \end{aligned} \end{equation}
We then have the following estimate of the spatial part of the Wasserstein metric:
\begin{align*}\frac \dd{\dd t}\int_{\R^{2d} \times S^2}|x - y|^2\dd\Pi_t(x,y,s_1,s_2)& \leq  -m\int_{\R^{2d} \times S^2}|x - y|^2 \dd\Pi_t(x,y,s_1,s_2) + \\ &+ \frac 1m\int_{\R^{2d} \times S^2}|\nabla f(x,s_1) - \nabla F(x)|^2\dd\Pi_t(x,y,s_1,s_2). \end{align*}
In particular, we can choose $\Pi_0 \coloneqq \rho_0 \otimes (\eta_0 \otimes \delta_{s_1 - s_2})$ and therefore, by uniqueness, $\Pi_t = \rho_t \otimes (\eta_t \otimes \delta_{s_1 - s_2})$. With this particular choice of $\Pi$:
\begingroup
\allowdisplaybreaks
\begin{align*}\int_{\R^{2d} \times S^2}|\nabla f(x,s_1) - &\nabla F(x)|^2 \dd\Pi_t(x,y,s_1,s_2) = \int_{\R^{2d} \times S^2}|\nabla f(x,s_2)- \nabla F(x) |^2\dd\Pi_t(x,y,s_1,s_2) \\ 
&\leq \int_{\R^{2d} \times S^2}|\nabla f(x,s_2) - \nabla F(x)|^2\,\dd \left(\int_S \rho_t(s_1)\dd s_1\right)(x)\dd \mu(s_2) \leq \sigma^2.
\end{align*}
\endgroup
We can put together the two previous inequalities with the Grönwall inequality:
$$\int_{\R^{2d} \times S^2}|x - y|^2 \dd\Pi_t(x,y,s_1,s_2) \leq \left(\int_{\R^{2d} \times S^2}|x - y|^2\dd\Pi_0(x,y,s_1,s_2)\right)e^{-mt} + \frac{\sigma^2}m\frac{1 - e^{-mt}}m$$
For the random part of our Wasserstein metric, we exploit the optimality of $\Gamma$:
\begin{align*}\frac \dd{\dd t}&\int_{\R^{2d} \times S^2}|s_1 - s_2|^2\dd\Pi_t(x,y,s_1,s_2)  = \\ &=
K\int_{\R^{2d} \times S^2}|s_1 - s_2|^2\left(\int_{S^2}\dd\Pi_t(x,y,s_1',s_2')\right)\dd\Gamma(s_1,s_2) - K\int_{\R^{2d} \times S^2}|s_1 - s_2|^2\dd\Pi_t(x,y,s_1,s_2) = \\& =
K\int_{S^2}|s_1 - s_2|^2\dd\Gamma(s_1, s_2) - K\int_{\R^{2d} \times S^2}|s_1 - s_2|^2\dd\Pi_t(x,y,s_1,s_2) = \\ &=
- K\int_{\R^{2d} \times S^2}|s_1 - s_2|^2\dd\Pi_t(x,y,s_1,s_2).
\end{align*}
Then, we have by definition the estimate  \eqref{eq:grazingestimate1} for the complete Wasserstein metric. \end{proof}

The above estimate is not optimal in the sense that there is still a variance term left as $K \rightarrow \infty$. If there were a probability measure $\Pi'_t$ such that $\Pi_t^K \overset{w}{\rightharpoonup} \Pi'_t$, where $\Pi^K$ is the solution of \eqref{coup_eq2} with the corresponding value of $K$, when $K \to \infty$, the variance term in \eqref{var_term} can be estimated better as
\begin{align*}\int_{\R^{2d} \times S^2}\Bigg(\nabla f(x, s_1) - &\int_S \nabla f(x,s')\,\dd \mu(s')\Bigg)\cdot (x - y)\dd(\Pi_t^K - \Pi'_t)(x,y,s_1,s_2) + \\ &+ \int_{\R^{2d} \times S^2}\left(\nabla f(x, s_1) - \int_S \nabla f(x,s')\,\dd \mu(s')\right)\cdot(x - y)\dd \Pi'_t(x,y,s_1,s_2). \end{align*}
\begin{figure}
    \centering
    \begin{subfigure}[b]{0.43\textwidth}
        \centering
        \includegraphics[width=\textwidth]{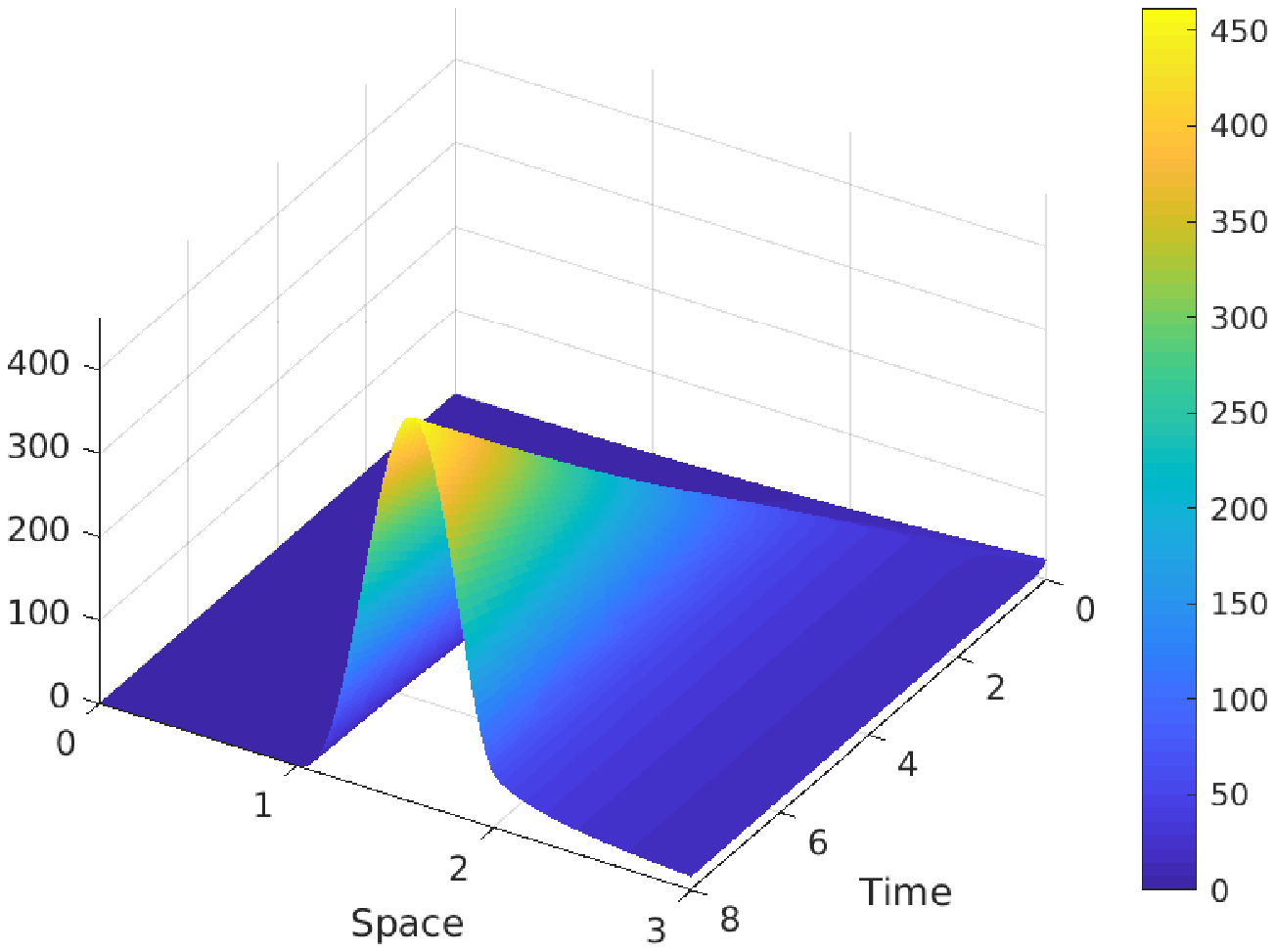}
    \end{subfigure}
    \begin{subfigure}[b]{0.43\textwidth}
        \centering
        \includegraphics[width=\textwidth]{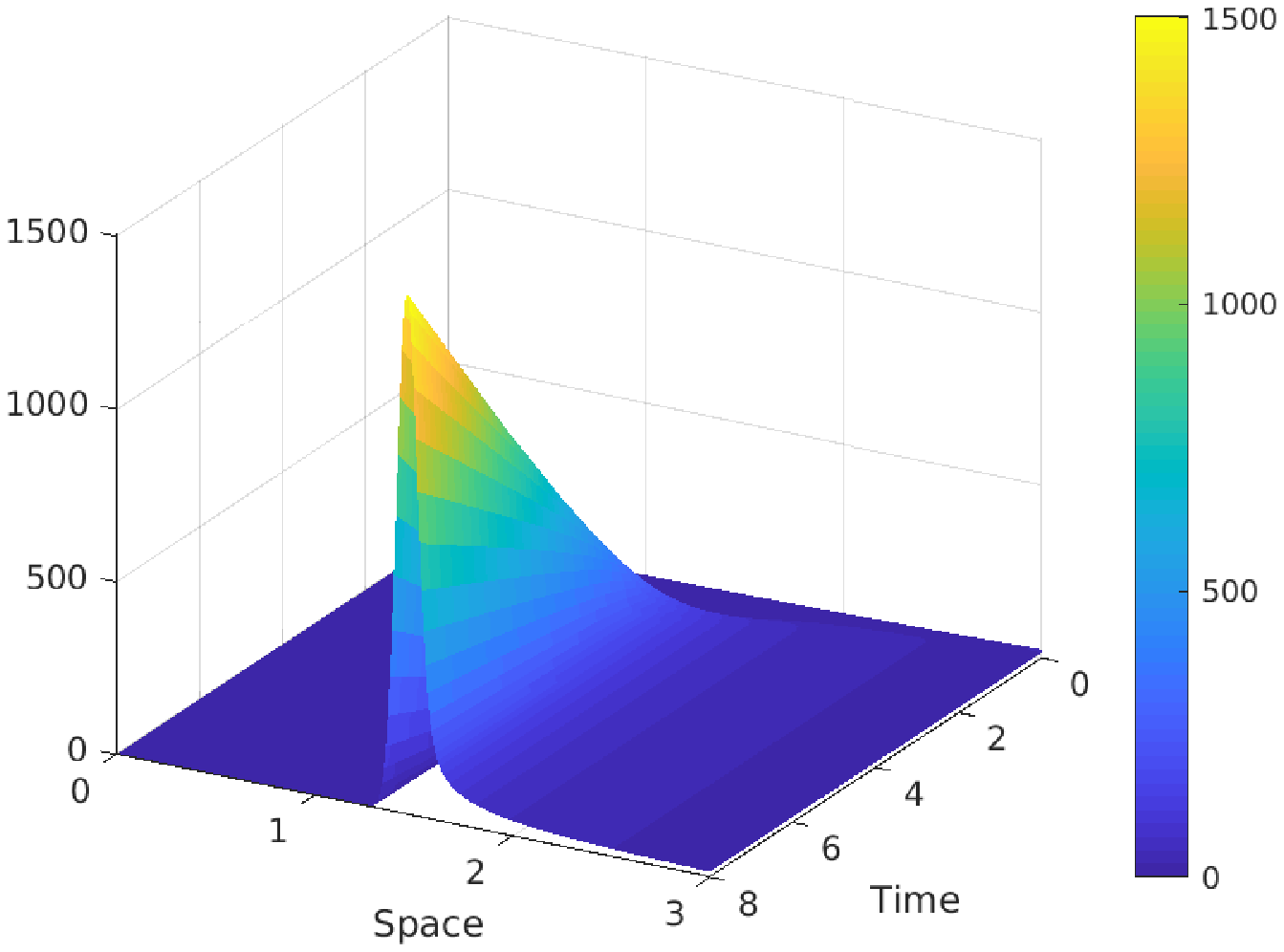}
    \end{subfigure}
    \begin{subfigure}[b]{0.43\textwidth}
        \centering
        \includegraphics[width=\textwidth]{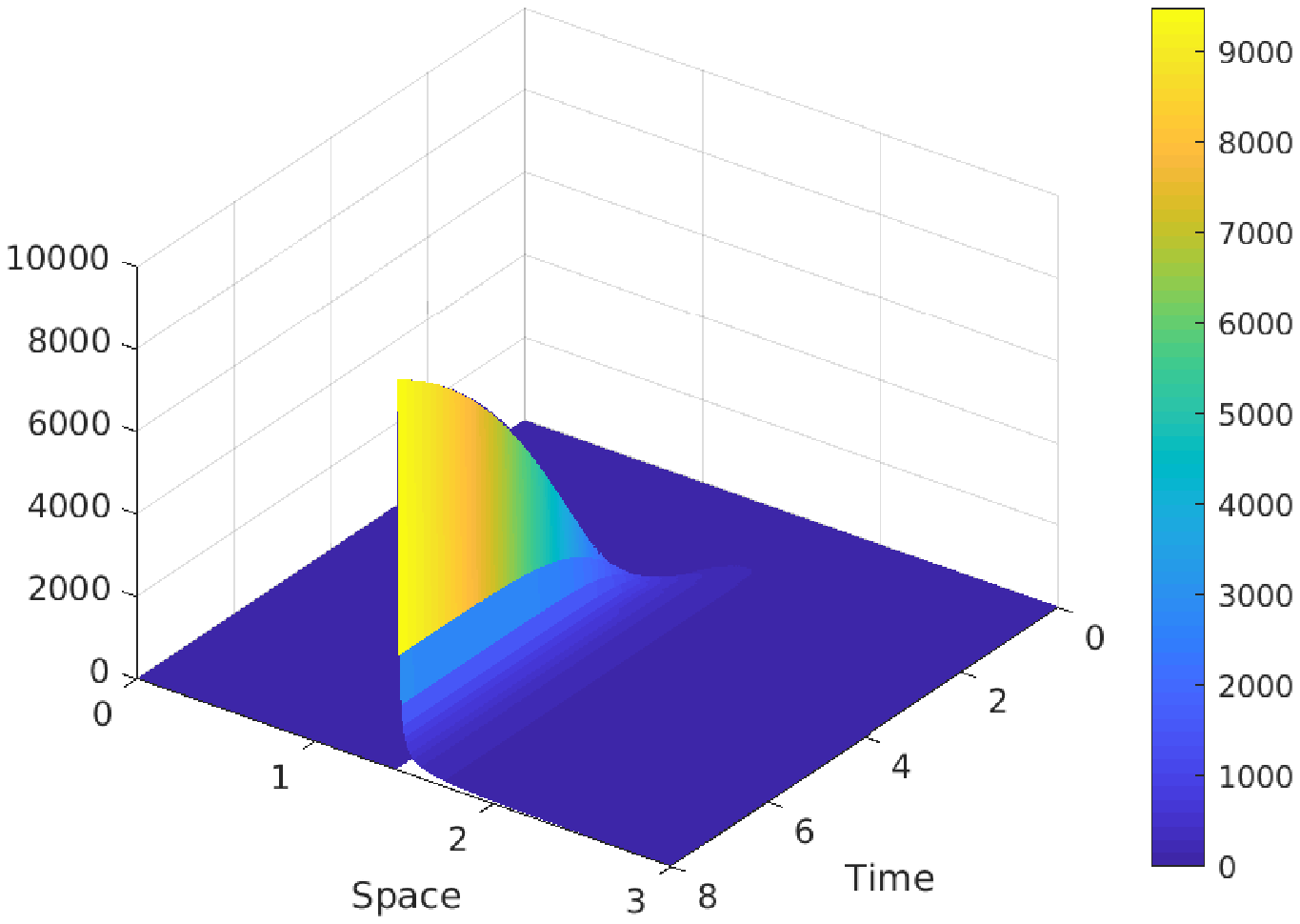}
    \end{subfigure}
    \caption{Solutions for a fixed label to equation \eqref{fund_eq_meas} with $f(x, s) = \frac s2|x - s|^2$: in the first two figures the solutions with $K = 10$ and $K = 10^9$, in the figure below the corresponding solution of the gradient flow as discussed in subsection \ref{graz_coll}}
    \label{conv_grad_flow}
\end{figure}
The first term can be treated this way:
\begin{align*} \Bigg|\int_{\R^{2d} \times S^2}\left(\nabla f(x, s_1) - \int_S \nabla f(x,s')\,\dd \mu(s')\right)&\cdot (x - y)\dd(\Pi_t^K - \Pi'_t)(x,y,s_1,s_2)\Bigg| \\ & \leq C\int_{\R^{2d} \times S^2}|x - y|\dd(\Pi_t^K - \Pi'_t)(x, y, s_1, s_2). \end{align*}
This weak convergence can be represented, by Skorohod's theorem (cf. section $17.3$ in \cite{PM91}), as an a.s. convergence of random variables $X_K$, $Y_K$, distributed as $\int_{\R^d \times S^2}\,\dd\Pi_t^K(y, s_1, s_2)$ and $\int_{\R^d \times S^2}\,\dd\Pi_t^K(x, s_1, s_2)$. Moreover, the second moment of the random variable $X_K - Y_K$ is uniformly bounded in $K$, thanks to the calculation for boundedness of the $2$-Wasserstein metric, and this tells us that the first term is approaching zero when $K \to \infty$, through uniform integrability in $K$ of $X_K - Y_K$. \\
Explicitly, the probability measure $\Pi'_t$ has the form $\Pi'_t = \Sigma_t(x, y)\Gamma(s_1, s_2) = \Sigma_t(x, y)\mu(s_1)\delta_{s_1 - s_2}$, where the probability measure $\Sigma_t(x, y) \coloneqq \int_{S^2}\Pi_t^K(s_1, s_2)$ is the marginal of $\Pi_t^K$ on $\R^{2d}$, which is independent of $K$, and hence we have that the second term is zero:
$$\int_{\R^{2d} \times S^2}\left(\nabla f(x, s_1) - \int_S \nabla f(x,s')\,\dd \mu(s')\right)\cdot(x - y)\dd \Pi'_t(x,y,s_1,s_2) = 0. $$
Now, integrating the inequalities for the spatial and the random parts of the Wasserstein metric, we get, by defining the function 
$$g_K(t) \coloneqq \int_{\R^{2d} \times S^2}\left(\nabla f(x, s_1) - \int_S \nabla f(x,s')\,\dd \mu(s')\right)\cdot (x - y)\dd(\Pi_t^K - \Pi'_t)$$ such that for every $t \in [0, T]$ with $T > 0$ $g_K(t) \to 0$ as $K \to \infty$, that:
\begin{align*}\int_{\R^{2d} \times S^2}&|x - y |^2 + |s_1 - s_2|^2\,\dd\Pi_t^K(x, y, s_1, s_2) \leq \left(\int_{\R^{2d} \times S^2}|x - y|^2\,\dd\Pi_0(x, y, s_1, s_2)\right)e^{-mt} + \\
& + \left(\int_{\R^{2d} \times S^2}|s_1 - s_2|^2\,\dd\Pi_0(x, y, s_1, s_2)\right)e^{-Kt} + \int_0^t e^{m(r - t)}g_K(r)\,\dd r.
\end{align*}
Choosing the initial data appropriately to nullify the first term at right hand side (i.e. such that $x = y$ $\Pi_0$-a.s.), we have that, when $K \to \infty$, $W_2(\rho_t, \eta_t \otimes \mu) \to 0$, for every $t > 0$, since $g_K = g_K(t)$ is integral of random variables with second moments bounded for $t \in [0, T]$, $K \geq 0$ and we can apply the Lebesgue theorem, that is, the solution of equation \eqref{fund_eq_meas} tends to the solution of the gradient flow $\partial_t \eta(t, y) = \nabla \cdot (\eta(t,y)\nabla F(y))$. \\
We miss just the proof of $\Pi_t^K \overset{w}{\rightharpoonup} \Pi'_t$, but the set of probability measures $(\Pi_t^K)_{K \geq 0}$ is tight for every $t \in [0, T]$, since the second moments are bounded in time and parameter $K$, so for every subsequence of $(\Pi_t^K)_{K \geq 0}$, there exists a subsubsequence which converges weakly to a certain probability measure $p_t$. One can notice that the following relation holds:
\begin{equation}\begin{aligned}\label{rel_prob} \Pi_t' = \Gamma(&s_1, s_2)\int_{S^2}\Pi_t'(x, y, s_1', s_2')\,\dd s_1'\dd s_2' = \Gamma(s_1, s_2)\Sigma_t(x, y) = \\&= \int_{S^2}\mu(s_1)\delta_{s_1 - s_2}\Pi_t^K(x, y, s_1', s_2')\,\dd s_1'\dd s_2' = \int_{S^2}\Gamma(s_1, s_2)\Pi_t^K(x, y, s_1', s_2')\,\dd s_1' \dd s_2'. \end{aligned}\end{equation}
We can multiply equation \eqref{coup_eq2} with $\phi \in C_b(S^2, \D(\R^{2d})) \subseteq C_b(S^2 \times \R^{2d})$ and, then, integrate on $\R^{2d} \times S^2$ and from $0$ to $T$, taking into account \eqref{rel_prob} and getting, for every $T > 0$: 
\begin{align*}\left|\int_0^T\left(\int_{\R^{2d} \times S^2}\phi\,\dd(\Pi_t^K - \Pi_t')\right)\dd t\right|& \leq \frac 1K\left|\int_{\R^{2d} \times S^2}\phi\,\dd(\Pi_T^K - \Pi_0)\right| + \\ &+ \frac 1K\int_0^T\int_{\R^{2d} \times S^2}|\nabla_x \phi \cdot \nabla_x f(x, s_1) + \nabla_y \phi \cdot \nabla_y F(y)|\,\dd \Pi_t^K(x,y,s_1,s_2), \end{align*}
and the terms at right hand side are bounded, thanks to the properties of $\phi$ and therefore, when $K \to \infty$, the term on left hand side tends to zero. This holds for every $T > 0$ and then we have that $\int_{\R^{2d} \times S^2}\phi \,\dd\Pi_t^K \to \int_{\R^{2d} \times S^2}\phi\,\dd \Pi_t', \forall \phi \in C_b(S^2, \D(\R^{2d}))$. But for the extracted subsubsequence, it holds $\int_{\R^{2d} \times S^2}\phi \,\dd\Pi_t^K \to \int_{\R^{2d} \times S^2}\phi\,\dd p_t, \forall \phi \in C_b(S^2, \D(\R^{2d}))$, so $p_t = \Pi_t'$, since $C_b(S^2, \D(\R^{2d}))$ is a separating class, and by Urysohn property we have the entire sequence converging to $\Pi_t'$.

 In order to gain a more precise rate of convergence, we can exploit the Young inequality and the Lipschitz continuity of the gradient:
\begin{align*}\frac \dd{\dd t}&\int_{\R^{2d} \times S^2}|x - y|^2\dd \Pi_t(x,y,s_1,s_2) = -2\int_{\R^{2d} \times S^2}(x - y)\cdot\left(\nabla_xf(x, s_1) - \nabla_y F(y)\right)\,\dd \Pi_t  \\ &= -2\int_{\R^{2d} \times S^2}(x - y)\cdot\left(\nabla_x F(x) - \nabla_y F(y) + \nabla_x f(x, s_1) - \nabla_x F(x) \right)\,\dd \Pi_t(x, y, s_1, s_2) \\ &\leq -2m\int_{\R^{2d}\times S^2}|x - y|^2\,\dd \Pi_t + m\int_{\R^{2d}\times S^2}|x - y|^2\,\dd \Pi_t + \frac {L^2}m\int_{\R^{2d}\times S^2}|s_1 - s_2|^2\,\dd \Pi_t \\ &\leq -m\int_{\R^{2d}\times S^2}|x - y|^2\,\dd \Pi_t(x, y, s_1, s_2) + \frac{L^2}m \left(\int_{\R^{2d}\times S^2}|s_1 - s_2|^2\,\dd \Pi_0(x, y, s_1, s_2)\right)e^{-Kt}.  \end{align*}
We can apply the Grönwall lemma to the previous inequality, getting:
\begin{align*}\int_{\R^{2d} \times S^2}|x - y|^2\dd \Pi_t(x,y,s_1,s_2) \leq &\left(\int_{\R^{2d} \times S^2}|x - y|^2\dd \Pi_0\right)e^{-mt} + \\ &+ \frac {L^2}m\left(\int_{\R^{2d} \times S^2}|s_1 - s_2|^2\dd \Pi_0\right)\frac{e^{-mt} - e^{-Kt}}{K-m}, \end{align*}
where we assumed $K > m$, which is justified for the limit of large $K$ we are interested in. 
Since the initial datum is chosen appropriately to nullify the first term at right hand side, that is $\Pi_0(x, y, s_1, s_2) = \delta_{x - y}(x)\otimes \mu(s_1) \otimes \eta_0(y) \otimes \mu(s_2)$, we obtain the rate of convergence in the Wasserstein metric:
$$W_2(\rho_t, \eta_t\otimes \mu) \leq C\left(L\,\sqrt{\frac{e^{-mt} - e^{-Kt}}{mK-m^2}} + e^{- \frac K2 t}\right),$$
where $C \coloneqq (\int_{\R^{2d} \times S^2}|s_1 - s_2|^2\,\dd \Pi_0(x, y, s_1, s_2))^{1/2}$.
\subsection{Concentration of solutions}
If $X \coloneqq (X_t, S_t)_{t \geq 0}$ is a stochastic process on $\R^d \times S$, such that $(X_t, S_t)$ is a random variable with law $\rho_t$,
we would like to estimate the average distance between $X_t$ and $x^*$.
In the special case $f$ $m$-strongly convex with $x^*$ unique minimizer for every $s \in S$ (and therefore $F$ is m-strongly convex as well), one has the estimate:
\begin{align*}\frac{\dd}{\dd t}\int_{\R^d \times S}|x - x^*|^2\,\dd\rho_t(x, s) &= -2\int_{\R^d \times S}(x - x^*)\cdot(\nabla f(x, s) - \nabla f(x^*, s))\,\dd\rho_t(x, s) \leq \\ &\leq -2m\int_{\R^d \times S}|x - x^*|^2\,\dd\rho_t(x, s)
\end{align*}
that is $\E \left[|X_t - x^*|^2\right] \leq \E \left[|X_0 - x^*|^2\right]e^{-2mt}$, i.e. $X_t$ converges to $x^*$ in $L^2$ as $t \to \infty$.
This can also give an idea about the convergence of $F(X_t)$ to $F(x^*)$. Since $F$ is also strongly convex, we find
%have the basic estimate:
%$$F(x^*) \geq F(X_t) + \nabla F(X_t)\cdot(x^* - X_t) + m|X_t - x^*|^2,$$
%and $x^*$ is a minimum for $F$, hence:
$$|F(X_t) - F(x^*)| \leq \nabla F(X_t)\cdot(X_t - x^*) - m|X_t - x^*|^2.$$
We integrate the previous inequality and we exploit the Lipschitz continuity of $\nabla F$:
\begin{equation*}\E\left[|F(X_t) - F(x^*)|\right] \leq \E\left[\left|\nabla F(X_t)\right|\cdot |X_t - x^*| - m|X_t - x^*|^2\right] \leq (L- m)\E\left[|X_t - x^*|^2\right], \end{equation*}
that is, in the particular case in which $f$ is strongly convex, $F(X_t)$ is converging exponentially in $L^1$ to $F(x^*)$ as $t \to \infty$. 

In general, if just $F$ is $m$--strongly convex, the previous calculation is not valid any longer and the estimates we can obtain are not that straightforward. For this sake we rescale time  with factor $\tau \coloneqq Kt$ and 
\begin{equation}\label{fund_eqresc} \frac{\partial \nu}{\partial t} = \frac{1}K \nabla \cdot(\nu \nabla f(x, s)) +  \int_S(\nu(x, s') - \nu(x, s))\dd \mu(s'),\end{equation}
and consider $X_t,S_t$ distributed according to the law $\nu_t$.

\begin{thm}\label{long_ineq} Let $F$ be $m$--strongly convex and $C >0$ sufficienly large.
Then there exist $K_0 > 0$  such that for all $K \geq K_0$ the estimate
\begin{equation}\begin{aligned}  &\E \left[\left  |X_t - x^* - \frac 1K\left(\nabla f(X_t, S_t) - \nabla F(X_t)\right)\right|^2\right] \leq \\
  &\E \left[\left  |X_0 - x^* - \frac 1K\left(\nabla f(X_0, S_0) - \nabla F(X_0)\right)\right|^2 + \frac C{K^2}|\nabla f(X_0, S_0) - \nabla F(X_0)|^2\right]e^{-\frac \beta K t}  + \\ & + \frac{1 - e^{-\frac \beta K t}}\beta\frac{C + 1}{K^2} \sigma^2  \label{mean_dis}
\end{aligned}\end{equation}
holds.
\end{thm}
\begin{proof}
With the parameter $K$ %and $t$ 
large enough, we exploit equation \eqref{fund_eq_meas} and the boundedness of the Hessian matrix of $f$ (and, therefore, of $F$), given by the Lipschitz continuity of $\nabla f$, hence $G \coloneqq \nabla^2 f - \nabla^2 F$ is bounded. In order to estimate this distance, we calculate, for $C > 0$ large enough:
\begingroup
\allowdisplaybreaks
\begin{align*}&\frac 1{2K}\frac \dd{\dd t}\left(\E \left[\left  |X_t - x^* - \frac 1K\left(\nabla f(X_t, S_t) - \nabla F(X_t)\right)\right|^2\right] + \frac C{K^2}\E\left[|\nabla f(X_t, S_t) - \nabla F(X_t)|^2\right]\right) = \\
&= \frac 1{2K}\frac \dd{\dd t}\int_{\R^d \times S}\left[\left|x - x^* - \frac 1K(\nabla f(x, s) - \nabla F(x))\right|^2 + \frac C{K^2}|\nabla f(x, s) - \nabla F(x)|^2\right]\dd \rho_t(x,s) = \\
&= -\frac 1{2K}\int_{\R^d \times S} \nabla\left|x - x^* - \frac 1K(\nabla f(x, s) - \nabla F(x))\right|^2\cdot \nabla f(x, s)\,\dd \rho_t(x,s) + \frac 12\int_{\R^d \times S} \Big|x - x^* - \\ &\qquad \frac 1K(\nabla f(x, s) - \nabla F(x))\Big|^2\,\dd(\bar \rho_t \otimes \mu - \rho_t)(x, s) - \frac C{2K^3}\int_{\R^d \times S}\nabla |\nabla f(x, s) - \nabla F(x)|^2\cdot \\ &\qquad \cdot \nabla f(x, s)\,\dd \rho_t(x, s) + \frac C{2K^2}\int_{\R^d \times S} |\nabla f(x, s) - \nabla F(x)|^2 \,\dd(\bar \rho_t \otimes \mu - \rho_t)(x, s) = \\ 
&= -\frac 1K\int_{\R^d \times S}\nabla f(x, s)\cdot\left(I_d + \frac GK\right)\left(x - x^* - \frac 1K(\nabla f(x, s) - \nabla F(x))\right)\,\dd \rho_t(x, s) + \\ &\qquad + \frac {C + 1}{2K^2}\int_{\R^d \times S} |\nabla f(x, s) - \nabla F(x)|^2 \,\dd(\bar \rho_t \otimes \mu - \rho_t)(x, s) - \frac 1K\int_{\R^d \times S}(x - x^*)\cdot(\nabla f(x, s) - \\ &\qquad -\nabla F(x))\,\dd(\bar \rho_t \otimes \mu - \rho_t)(x,s) - \frac C{K^3}\int_{\R^d \times S} \nabla f(x, s)\cdot G(\nabla f(x, s) - \nabla F(x))\,\dd \rho_t(x, s).
\end{align*}
We can then use the definition of $\sigma^2$ on the second term, exploit that $\nabla F(x^*) = 0$ and add and subtract a term with $\nabla F(x)$. In this way, the quantity above can be estimated by:
\begin{align*}
& -\frac 1K\int_{\R^d \times S}(\nabla f(x, s) - \nabla F(x))\cdot\left(I_d + \frac GK\right)\left(x - x^* - \frac 1K(\nabla f(x, s) - \nabla F(x))\right)\,\dd \rho_t(x, s) - \\ &\qquad -\frac 1K\int_{\R^d \times S}(\nabla F(x) - \nabla F(x^*))\cdot\left(I_d + \frac GK\right)\left(x - x^* - \frac 1K(\nabla f(x, s) - \nabla F(x))\right)\,\dd \rho_t(x, s) + \\ &\qquad + \frac{C + 1}{2K^2}\sigma^2 - \frac {C + 1}{2K^2}\int_{\R^d \times S} |\nabla f(x, s) - \nabla F(x)|^2 \,\dd\rho_t(x, s) + \frac 1K\int_{\R^d \times S}(x - x^*)\cdot(\nabla f(x, s) - \\ &\qquad -\nabla F(x))\,\dd\rho_t(x,s)  - \frac C{K^3}\int_{\R^d \times S} (\nabla f(x, s) - \nabla F(x))\cdot G(\nabla f(x, s) - \nabla F(x))\,\dd \rho_t(x, s) - \\ &\qquad  - \frac C{K^3}\int_{\R^d \times S} (\nabla F(x) - \nabla F(x^*))\cdot G(\nabla f(x, s) - \nabla F(x))\,\dd \rho_t(x, s) \\
&\leq \left(\frac 1{K^2} - \frac{C + 1}{2K^2}\right)\int_{\R^d \times S} |\nabla f(x, s) - \nabla F(x)|^2 \,\dd\rho_t(x, s) - \frac 1{K^2}\int_{\R^d \times S}(\nabla f(x, s) - \nabla F(x) + \nabla F(x) - \\ &\qquad -\nabla F(x^*))\cdot G\left(x - x^* - \frac 1K(\nabla f(x, s) - \nabla F(x))\right)\,\dd \rho_t(x, s) -\frac mK \int_{\R^d \times S}|x - x^*|^2\,\dd \rho_t(x, s) + \\ &\qquad + \frac 1{K^2}\int_{\R^d \times S}(\nabla F(x) - \nabla F(x^*))\cdot (\nabla f(x, s) - \nabla F(x))\,\dd \rho_t(x, s) + \frac{C + 1}{2K^2}\sigma^2 - \\ &\qquad -  
\frac C{K^3}\int_{\R^d \times S} (\nabla f(x, s) - \nabla F(x))\cdot G(\nabla f(x, s) - \nabla F(x))\,\dd \rho_t(x, s) - \\ &\qquad - 
\frac C{K^3}\int_{\R^d \times S} (\nabla F(x) - \nabla F(x^*))\cdot G(\nabla f(x, s) - \nabla F(x))\,\dd \rho_t(x, s),
\end{align*}
where the convexity of $F$ is harnessed. If we define $C_G \coloneqq \|G\|_\infty$, one can estimate the term above by the Cauchy-Schwarz and Young inequalities:
\begin{align*}
& -\frac mK \int_{\R^d \times S}|x - x^*|^2\,\dd \rho_t(x, s) + \frac{C + 1}{2K^2}\sigma^2 + \frac {1 - C}{2K^2}\int_{\R^d \times S} |\nabla f(x, s) - \nabla F(x)|^2 \,\dd\rho_t(x, s) + \\ &\qquad + \frac {C_G}{K^2}\int_{\R^d \times S}\left(|\nabla f(x, s) - \nabla F(x)| + L|x - x^*|\right)\left(|x - x^*| + \frac 1K|\nabla f(x, s) - \nabla F(x)|\right)\,\dd \rho_t(x, s) + \\ &\qquad + \frac L{K^2}\int_{\R^d \times S}| x - x^*|\cdot |\nabla f(x, s) - \nabla F(x)|\,\dd \rho_t(x, s) + \frac {C}{K^3}\int_{\R^d \times S} |\nabla f(x, s) - \\ &\qquad -\nabla F(x)|^2 \,\dd\rho_t(x, s) + C\frac {LC_G}{K^3}\int_{\R^d \times S} |x - x^*|\cdot |\nabla f(x, s) - \nabla F(x)|\,\dd \rho_t(x, s)  \\
&\leq \left(- \frac mK + \frac {2LC_G + C_G +L}{2K^2} + \frac {LC_G + CLC_G}{2K^3}\right)\int_{\R^d \times S}|x - x^*|^2\,\dd \rho_t(x, s) +  \frac{C + 1}{2K^2}\sigma^2 +  \\ &\qquad +\left(\frac{1 - C}{2K^2} + \frac {C_G + L}{2K^2} + \frac{2C_G + LC_G + 2C + CLC_G}{2K^3}\right)\int_{\R^d \times S} |\nabla f(x, s) - \nabla F(x)|^2 \,\dd\rho_t(x, s).
\end{align*}
We then define $\alpha \coloneqq \min\{m - \frac {2LC_G + C_G +L}{2K} - \frac {LC_G + CLC_G}{2K^2}, \frac{C - 1}{2} - \frac {C_G + L}{2} - \frac{2C_G + LC_G + 2C + CLC_G}{2K}\}$ and $\alpha > 0$ if $K$ and $C$ are large enough. With this definition in mind, we can proceed:
\begin{align*}
&-\frac \alpha K\int_{\R^d \times S}|x - x^*|^2\,\dd \rho_t(x, s) - \frac \alpha {K^2}\int_{\R^d \times S} |\nabla f(x, s) - \nabla F(x)|^2 \,\dd\rho_t(x, s) + \frac{C + 1}{2K^2}\sigma^2  \\
&\leq -\frac {\alpha'}K\int_{\R^d \times S}|x - x^*|^2\,\dd \rho_t(x, s) - \frac{\alpha'}{K^3}\int_{\R^d \times S} |\nabla f(x, s) - \nabla F(x)|^2 \,\dd\rho_t(x, s) - \\ &\qquad - \frac {\alpha'}{K^3}\int_{\R^d \times S} |\nabla f(x, s) - \nabla F(x)|^2 \,\dd\rho_t(x, s) + \frac{C + 1}{2K^2}\sigma^2 \end{align*}
where $\alpha' \coloneqq \alpha/2$, since $1/K^3 \leq 1/K^2$ if $K$ is large enough. One can apply again the Young inequality:
\begin{align*}
&-\frac {\alpha'}{2K}\int_{\R^d \times S}|x - x^*|^2\,\dd \rho_t(x, s) - \frac{\alpha'}{2K^3}\int_{\R^d \times S} |\nabla f(x, s) - \nabla F(x)|^2 \,\dd\rho_t(x, s) + \frac{C + 1}{2K^2}\sigma^2 \\ &\qquad +\frac{\alpha'}{2K}\int_{\R^d \times S} 2(x - x^*)\cdot(\nabla f(x, s) - \nabla F(x)) \,\dd\rho_t(x, s) - \frac {\alpha'} C \frac C{K^3}\int_{\R^d \times S} |\nabla f(x, s) - \nabla F(x)|^2 \,\dd\rho_t(x, s) \\
&\leq -\frac{\alpha''}{2K}\int_{\R^d \times S}\left[\left|x - x^* - \frac 1K(\nabla f(x, s) - \nabla F(x))\right|^2 + \frac C{K^2}|\nabla f(x, s) - \nabla F(x)|^2\right]\,\dd \rho_t(x, s) + \frac{C + 1}{2K^2}\sigma^2
\end{align*}
\endgroup
where we applied again the Young inequality and defined $\alpha'' \coloneqq \min\{\alpha', \frac{2\alpha'}C\}$. Thanks to the Grönwall lemma applied to first and last terms of the previous long calculation, we proved the following inequality for $K$ large enough and a certain constant $\beta > 0$:
\begin{equation*}\begin{aligned}  &\E \left[\left  |X_t - x^* - \frac 1K\left(\nabla f(X_t, S_t) - \nabla F(X_t)\right)\right|^2\right] \leq \\
&\E \left[\left  |X_t - x^* - \frac 1K\left(\nabla f(X_t, S_t) - \nabla F(X_t)\right)\right|^2 + \frac C{K^2}|\nabla f(X_t, S_t) - \nabla F(X_t)|^2\right] \leq \\ &\E \left[\left  |X_0 - x^* - \frac 1K\left(\nabla f(X_0, S_0) - \nabla F(X_0)\right)\right|^2 + \frac C{K^2}|\nabla f(X_0, S_0) - \nabla F(X_0)|^2\right]e^{-\frac \beta K t}  + \\ & + \frac{1 - e^{-\frac \beta K t}}\beta\frac{C + 1}{K^2} \sigma^2. 
\end{aligned}\end{equation*}
\end{proof}

Estimate \eqref{mean_dis}, when $K$ is large, can give an idea about the concentration of $X_t$ around $x^*$, since the term with the gradients is bounded, due to the Lipschitz continuity of $\nabla f$. Moreover, it highlights the role of the variance term $\sigma^2$. 

%%%%%%%%%%%%%%%%%%%%%%%%%%%%%%%%%%%%%%%%%%%%%%%%%%%%%%
\subsection{Variable learning rate} \label{sec_var}

In the following we investigate the case of $K(t) \rightarrow \infty$ as $t \rightarrow \infty$, where we expect the convergence to a large-time limit concentrated at minimizers of $F$. Note that  $\frac 1K$ can be interpreted as the learning rate in a stochastic gradient descent iteration on $F$, since when $K \to \infty$, we have convergence towards the gradient flow.

We make the following assumptions in this section:
\begin{itemize}
    \item $F$ is a $m$--strongly convex function with unique minimizer $x^*$; 
    \item $K = K(t)$ is a continuously differentiable function, with $K > 0$ and $K' \geq \gamma > 0$;
    \item $\int_0^{+\infty} \frac 1{K(u)}\,\dd u = +\infty$ and $\int_0^{+\infty} \frac 1{K^2(u)}\,\dd u < +\infty$. 
\end{itemize}

Note that the conditions on $K$ can be translated again to those usually required for the learning rate of the stochastic gradient descent algorithm in the convergence analysis of the latter (cf. \cite{Turinici} for a simple exposition).

Since $K$ is regular enough, we have again well-posedness of  equation \eqref{fund_eq} with variable $K$ by analogous methods. If the stochastic process $(X_t, S_t)$ is distributed as $\rho_t$ and we rescale time appropriately, we can obtain the same estimate as in \eqref{mean_dis} about the average square distance between $X_t$ and $x^*$, which augment now with a higher order term due to the dependency of $K$ on time:

\begin{lemma}
Let $F$ and $K$ satisfy the above assumptions and let 
$$\varepsilon(t) \coloneqq \E \left[\left  |X_t - x^* - \frac 1{K(t)}\left(\nabla f(X_t, S_t) - \nabla F(X_t)\right)\right|^2\right] + \frac C{K(t)^2}\E\left[|\nabla f(X_t, S_t) - \nabla F(X_t)|^2\right]. $$
Then $\varepsilon(t) \rightarrow 0$ as $t \rightarrow \infty$.
\end{lemma}
\begin{proof}We can proceed as in theorem \ref{long_ineq}, obtaining more terms due to the dependency of $K$ on time:
\begin{equation}\begin{aligned} \label{aver_var}
\frac 12\frac {\dd \varepsilon}{\dd t} &\leq -\frac \alpha {K(t)} \varepsilon + \frac{C + 1}{2K^2(t)}\sigma^2 + \frac{K'(t)}{K^3(t)}\int_{\R^d \times S}\left(x - x^* + \frac 1K(\nabla f(x, s) - \nabla F(x)\right)\cdot \\ 
&\cdot (\nabla f(x, s) - \nabla F(x))\, \dd \rho_t(x, s) - \frac{CK'(t)}{K^4(t)}\int_{\R^d \times S} |\nabla f(x, s) - \nabla F(x)|^2 \,\dd\rho_t(x, s)  \\ &\leq -\frac {\alpha'}{K(t)}\varepsilon + \frac{C + 1}{2K^2(t)}\sigma^2,
\end{aligned}\end{equation}
where the latter holds if $t$ is large enough, thanks to Young's inequality. If we integrate \eqref{aver_var} from $T$ large enough to $t$, we obtain:
$$\varepsilon(t) + 2\alpha'\int_T^t \frac{\varepsilon (u)}{K(u)}\,\dd u \leq \varepsilon(T) + \sigma^2(C + 1)\int_T^t \frac 1{K^2(u)}\,\dd u$$
and due to the second condition on $K$, we have that $\int_0^{+\infty} \frac{\varepsilon (u)}{K(u)}\,\dd u < +\infty$ and, since $\int_0^{+\infty} \frac 1{K(u)}\,\dd u = +\infty$, we gain that $\varepsilon(t) \to 0$ if $t \to \infty$. 
\end{proof}
\begin{figure}
    \centering
    \begin{subfigure}[b]{0.49\textwidth}
        \centering
        \includegraphics[width=\textwidth]{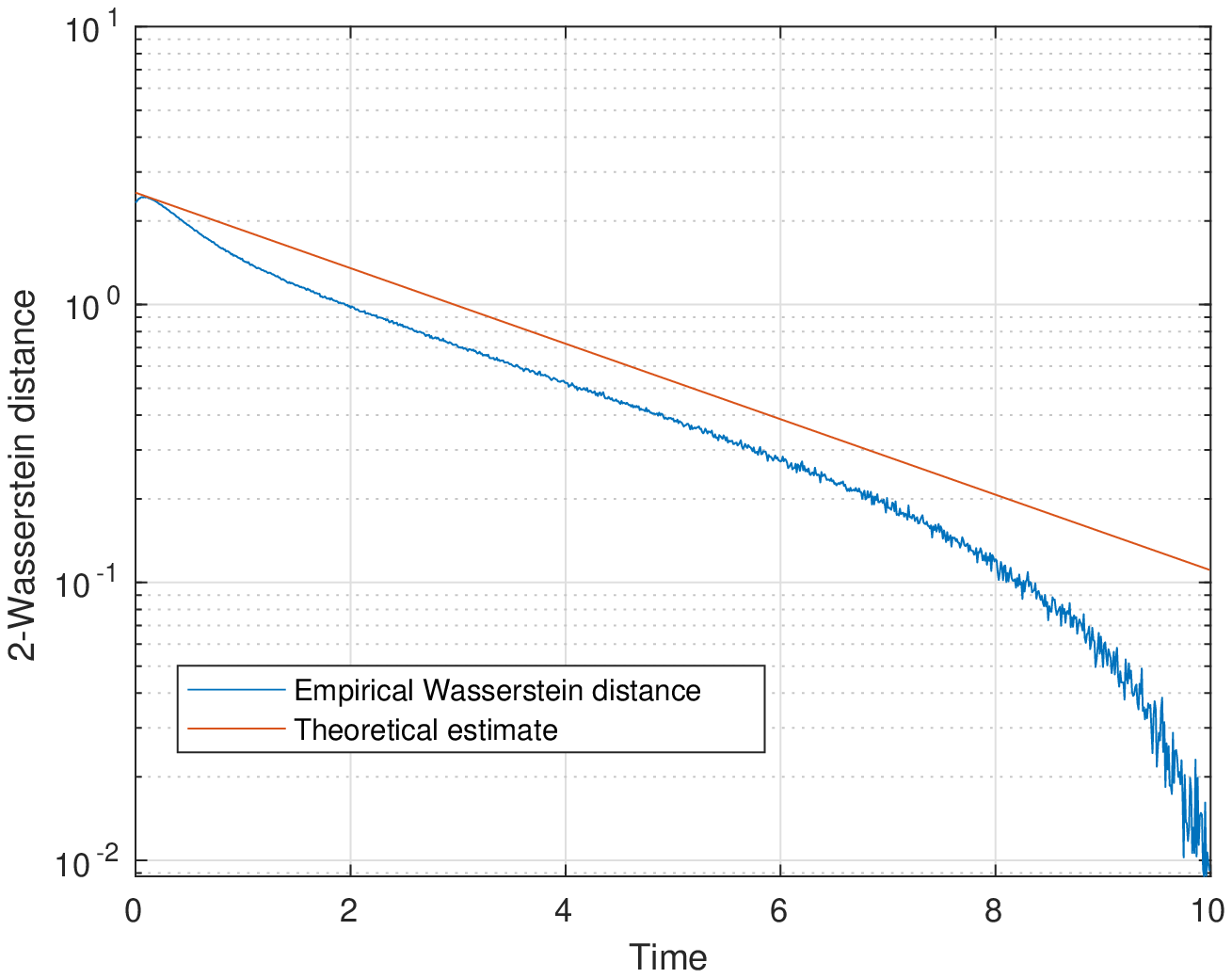}
    \end{subfigure}
    \begin{subfigure}[b]{0.49\textwidth}
        \centering
        \includegraphics[width=\textwidth]{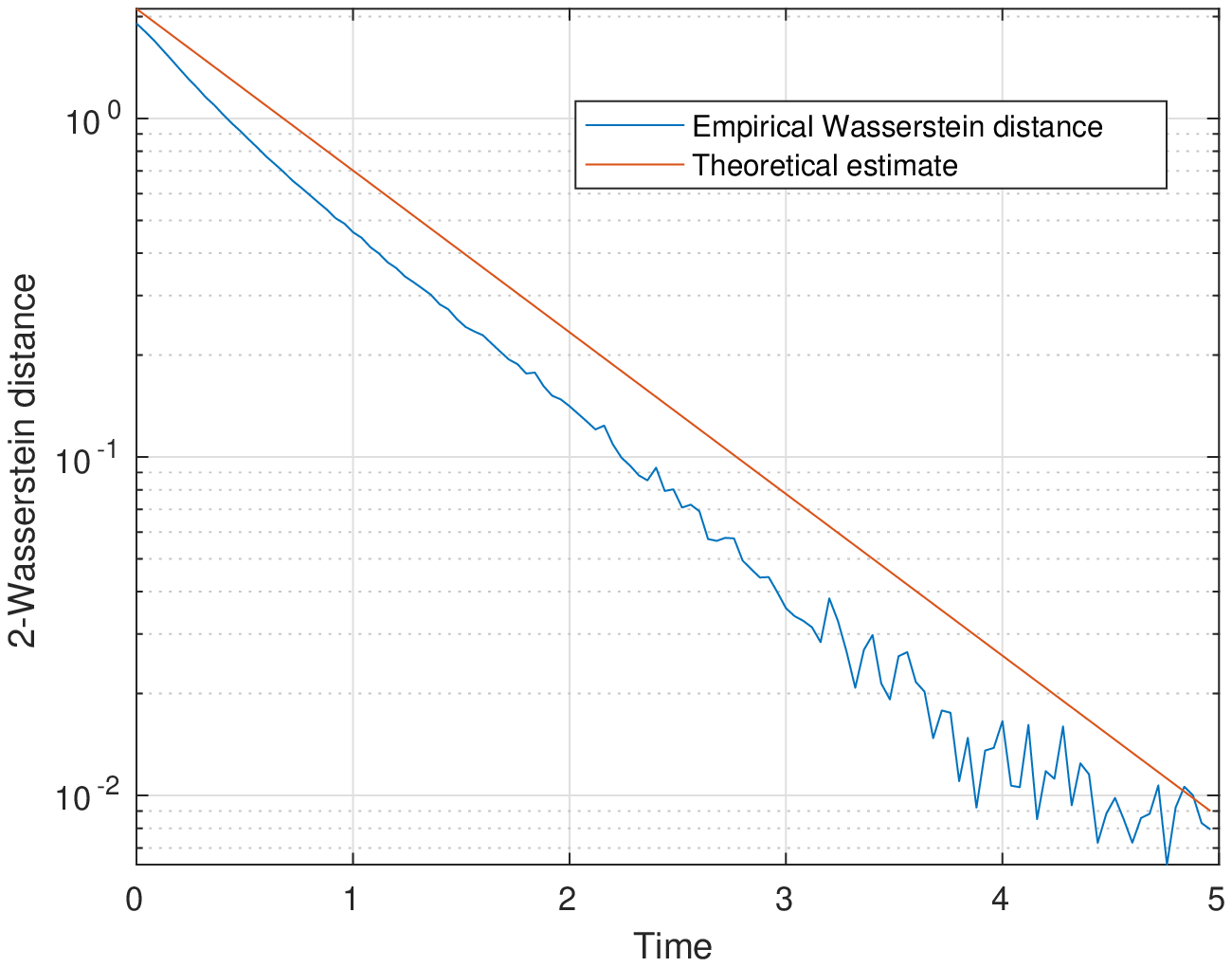}
    \end{subfigure}
    \caption{Convergence in semilogarithmic scale given by $d = 1$ with quadratic potentials $f(x, s) = \frac s2|x - s|^2$ and $f(x, s) = sx^2$, $K(t) = 1 + t$, $S = \{1, 2\}$ and $\mu = Bern(0.5)$ towards $\delta_{5/3}$ and $\delta_0$ compared with the theoretical estimate of subsection \ref{sec_var}.}
\end{figure}
This justifies the following estimate about the convergence of the solution towards the minimizer $x^*$ and this convergence can be estimated through the coupling equation:
\begin{equation} \begin{aligned}\label{coup_eq3}
     \frac{\partial \Pi_t}{\partial t}  =& \, \nabla_x \cdot(\Pi_t \nabla_x f(x, s_1)) +  \nabla_y \cdot(\Pi_t \nabla_y F(y)) 
     + \\ & K(t)~ \int_S \int_S \left(\Gamma(s_1, s_2)\delta_{x^*}(y)\left(\int_{\R^d}\Pi_t(x, y', s_1', s_2')\,\dd y'\right) - \Pi_t(x, y, s_1, s_2)\right)\dd s_1' \dd s_2',
 \end{aligned} \end{equation}
 whose marginals are \eqref{fund_eq}, whose solution is $\rho$, and 
 \begin{equation}\label{marg2} \partial_t \eta = \nabla_y \cdot(\eta \nabla_yF(y)) + K(\delta_{x^*}\otimes \mu - \eta). \end{equation} 
 A stationary solution of \eqref{marg2} is $\delta_{x^*}\otimes \mu$ and we choose as initial datum $\Pi_0$ a coupling of $\rho_0$, with marginal $\mu$ on $S$, and $\delta_{x^*}\otimes \mu$. Proceeding as usual, we get the estimates for the spatial part of the Wasserstein metric:
\begin{align*}\frac \dd{\dd t}&\int_{\R^{2d} \times S^2}|x - y|^2\,\dd \Pi_t(x, y, s_1, s_2) =
2\int_{\R^{2d} \times S^2}(x - y)\cdot(\nabla_y F(y) - \nabla_x f(x, s_1))\,\dd \Pi_t(x,y,s_1, s_2) + \\ &+ K(t)\int_{\R^{2d} \times S^2}|x - x^*|^2\,\dd\Pi_t(x, y, s_1, s_2) - K(t)\int_{\R^{2d} \times S^2}|x - y|^2\,\dd\Pi_t(x, y, s_1, s_2) = \\ &= 2\int_{\R^{2d} \times S^2}(x - y)\cdot(\nabla_y F(y) - \nabla_x f(x, s_1))\,\dd \Pi_t(x,y,s_1, s_2),
\end{align*}
since $y = x^*$ a.s., considering the marginals of $\Pi_t$.
For the random part:
$$\frac \dd{\dd t}\int_{\R^{2d} \times S^2}|s_1 - s_2|^2\,\dd \Pi_t(x, y, s_1, s_2) = -K(t)\int_{\R^{2d} \times S^2}|s_1 - s_2|^2\,\dd \Pi_t(x, y, s_1, s_2), $$
that is $\int_{\R^{2d} \times S^2}|s_1 - s_2|^2\,\dd \Pi_t(x, y, s_1, s_2) = \left(\int_{\R^{2d} \times S^2}|s_1 - s_2|^2\,\dd \Pi_0(x, y, s_1, s_2)\right)e^{-\int_0^t K(u)\,\dd u}$.
Using again the Young inequality and the properties of $\nabla f$ on the spatial part, we get:
\begin{align*}\frac \dd{\dd t}\int_{\R^{2d} \times S^2}|x - y|^2\,\dd \Pi_t(x, y, s_1, s_2) \leq -\frac m2 &\int_{\R^{2d} \times S^2}|x - y|^2\,\dd \Pi_t(x, y, s_1, s_2) + \\ &+\frac {cL}{2m} \int_{\R^{2d} \times S^2}|s_1 - s_2|^2\,\dd \Pi_t(x, y, s_1, s_2). \end{align*}
We can apply the Grönwall lemma to the previous inequality, gaining for a constant $c > 0$:
\begin{equation}\begin{aligned} \int_{\R^{2d} \times S^2}&|x - y|^2\,\dd \Pi_t(x, y, s_1, s_2) \leq \\ \leq &\left(\int_{\R^{2d} \times S^2}|x - y|^2\,\dd \Pi_0(x, y, s_1, s_2)\right)e^{-\frac m2 t} + ce^{-\frac m2 t}\int_0^t e^{\frac m2 s - \int_0^s K(u)\,\dd u}\,\dd s \end{aligned} \end{equation}
and for adequate choices of $K$, when $t \to \infty$, we have that $W_2(\rho_t, \delta_{x^*} \otimes \mu) \to 0$ exponentially. 
%%%%%%%%%%%%%%%%%%%%%%%%%%%%%%%%%%%%%%%%%%%%%%%%%%%%%%%%%%%%%%%%%%%%%%%%%%%%%%%%%%%%%%%%%%%%%%%
\section{Conclusion and further perspectives} \label{furt_per}
In this paper, we introduced a continuous time approximation that keeps the compact support of stationary solutions, we proved existence, uniqueness and gave examples of stationary states with compact support. Moreover, we proved stability with respect to the switching probability $\mu$, which is important when $\mu$ is replaced by empirical measures. We provided convergence analysis in Wasserstein spaces, which holds for $K$ constant and $K \to \infty$ as well as for variable learning rates. In this section, we also insert some open issues and further ideas about the model presented in the article.
\subsection{Hilbert expansion and stochastic modified equations}
In section \ref{conv_an} we investigated the first order convergence of equation \eqref{fund_eq_meas} when $K \to \infty$, but the higher order expansion remains unexplored. We can give some intuitive insights about it, exploiting again the Hilbert expansion \eqref{hilb_ex}. We already know that $\rho_0 = \bar \rho_0 \otimes \mu$ and the equation of higher order is:
\begin{equation} \label{first_or}\frac {\partial \rho_0}{\partial t} = \nabla \cdot(\rho_0 \nabla f) + \bar \rho_1 \otimes \mu - \rho_1. \end{equation}
This equation can be integrated on $S$, taking into account the structure of $\rho_0$:
$$\frac {\partial \bar \rho_0}{\partial t} = \nabla \cdot \left(\bar \rho_0 \nabla \int_S f\,\dd\mu(s)\right)= \nabla \cdot(\bar \rho_0 \nabla F)$$
that is $\partial_t \rho_0 = \nabla \cdot (\rho_0 \nabla F)$, which can be substituted in \eqref{first_or}, getting:
$$\rho_1 = \bar \rho_1 \otimes \mu + \nabla \cdot(\rho_0(\nabla f - \nabla F))$$
which can be in turn integrated against $\nabla f$:
\begin{equation}\label{eq_ut}\int_S \rho_1 \nabla f\,\dd\mu = \bar \rho_1\nabla F + \int_S \nabla \cdot(\rho_0 (\nabla f - \nabla F))\nabla f\,\dd \mu. \end{equation}
The last term can be analysed better:
\begin{align*} \int_S &\nabla \cdot(\rho_0 (\nabla f - \nabla F))\nabla f\,\dd \mu = \int_S \nabla \cdot(\rho_0 (\nabla f - \nabla F))(\nabla f - \nabla F)\,\dd \mu = \\ &= \int_S \nabla \cdot (\rho_0(\nabla f - \nabla F)\otimes(\nabla f - \nabla F))\,\dd \mu - \int_S \rho_0(\nabla^2 f - \nabla^2 F)(\nabla f - \nabla F)\,\dd \mu = \\ &= \nabla \cdot \left(\rho_0 \int_S (\nabla f - \nabla F)\otimes(\nabla f - \nabla F)\,\dd \mu \right) - \frac 12 \rho_0 \nabla \left(\int_S |\nabla f - \nabla F|^2\,\dd \mu\right).\end{align*}
The second order equation is $\partial_t \rho_1 = \nabla \cdot(\rho_1 \nabla f) + \overline \rho_2 \otimes \mu - \rho_2$, which, if integrated, yields:
\begin{equation}\label{eq_ut2}\frac {\partial \bar \rho_1}{\partial t} = \nabla \cdot\left(\int_S \rho_1\nabla f\,\dd \mu\right). \end{equation}
If $\Sigma \coloneqq \E_s[(\nabla f - \nabla F)\otimes(\nabla f - \nabla F)]$ and $V \coloneqq \E_s[|\nabla f - \nabla F|^2]$, then, considering \eqref{eq_ut} and \eqref{eq_ut2}, $\bar \rho$ solves:
\begin{align*}\partial_t \bar \rho &= \partial_t \left(\bar \rho_0 + \frac 1K \bar \rho_1\right) + O\left(\frac 1{K^2}\right) = \\ &= \nabla \cdot \left(\left(\bar \rho_0 + \frac 1K\bar \rho_1\right)\nabla F\right) + \frac 1K \nabla \cdot (\nabla \cdot (\bar \rho_0\Sigma)) - \frac 1{2K}\nabla \cdot (\bar \rho_0\nabla V) + O\left(\frac 1{K^2}\right)
\end{align*}
If, in particular, $\nabla F \equiv 0$ and $K \to \infty$, one obtains, rescaling time:
\begin{equation}\partial_t \bar \rho = \nabla \cdot \nabla \cdot(\bar \rho \Sigma) - \frac 12\nabla \cdot(\bar \rho \nabla V). \label{sec_ord_eq} \end{equation}
The condition $\nabla F \equiv 0$ is fulfilled if $f(x, s) = s \cdot x$, with $S = \S^{d-1}$ and $\mu = \frac 1{\CMcal H^{d - 1}(\S^{d - 1})}\CMcal{H}^{d -1}\mres \,\S^{d -1}$: in this case $\nabla f(x, s) = s$ and we gain that $\partial_t \bar \rho = \Delta \bar \rho$, i.e. $\bar \rho$ satisfies the heat equation. With this particular choice of the potential, the problem of convergence towards equation \eqref{sec_ord_eq}, which we do not tackle in this paper, is simply reduced to the convergence of the linear Boltzmann equation towards the heat equation.

If $F \neq 0$ we obtain an equation in the original time with a higher order term, which is reminiscent of the stochastic modified equations (cf. \cite{FGL19,LTW19}) 
\begin{equation}\partial_t \bar \rho = \nabla \cdot ( \bar \rho \nabla f) + \frac{1}K \nabla \cdot \nabla \cdot(\bar \rho \Sigma) - \frac 1{2K}\nabla \cdot(\bar \rho \nabla V).  \end{equation}

%%%%%%%%%%%%%%%%%%%%%%%%%%%%%%%%%%%%%%%%%%%%%%%%%%%%%%%%%%%%%%%%%%%%%%%%%%%%%%%
\subsection{Diffusion}

A natural modification would be to include entropic parts in the energies, possibly with a coefficient $D$ depending on $s$. This yields an equation with additional diffusion of the form
\begin{equation} \frac{\partial \rho}{\partial t} = \nabla \cdot(\rho \nabla f(x, s)) + D(s) \Delta\rho + K  (\bar{\rho}\otimes \mu - \rho ),\end{equation}
which can indeed be analyzed in the same way of \eqref{fund_eq_meas}. The main change appears in the existence proof, which needs to be changed to using techniques for parabolic systems. As long as $D$ is uniformly positive in $s$, this is a rather straight-forward task. The existence proof for stationary solutions is hardly changing, since we use a diffusive approximation anyway. The coupling estimates can be handled in the same way as in \cite{FP19}. With the couplings for the diffusive part constructed there it is even possible to obtain estimates for equations of the form
\begin{equation} \frac{\partial \rho}{\partial t} = \nabla \cdot(\rho \nabla f(x, s)) + \nabla \cdot ( D(x,s) \nabla \rho) + K  (\bar{\rho}\otimes \mu - \rho ),\end{equation}
with $D$ being a positive definite diffusion matrix. The estimates will however have additional exponentially growing terms in time if $D$ is not independent of $x$, or otherwise one would need to find more appropriate transport costs (cf. \cite{FP19}).

%%%%%%%%%%%%%%%%%%%%%%%%%%%%%%%%%%%%%%%%%%%%%%%%%%%%%%%%%%%%%%%%%%%%%%%%%%%
\subsection{Mean--field equations and neural networks} 
We would like to introduce in this subsection a possible nonlinear generalization to equation \eqref{fund_eq_meas} as mean--field limit of a hierarchy of equations given by the loss with respect to a neural network. \\ For these reasons, we consider a shallow neural network $\Phi_x(s_1)  = \frac 1N\sum_{i = 1}^Nx_i^3\sigma(x_i^1 \cdot s_1 + x_i^2)$ with data space $S \subseteq \R^{d + 1}$, $s = (s_1, s_2) \in S$, $x = (x_1, \dots, x_N) \in \R^{N(d + 2)}$ parameters of the neural network and activation function $\sigma$. We fix the quadratic loss $f(x, s) \coloneqq \frac N2|\Phi_x(s_1) - s_2|^2$, of which we can work out the gradient:
\begin{equation*}\begin{aligned} \nabla_{x_i}f(x, s) &= N(\Phi_x(s_1) - s_2)\nabla_{x_i}\Phi_x(s_1) = \\ &= \left(\frac 1N\sum_{j = 1}^N x_j^3\sigma(x_j^1\cdot s_1 + x_j^2)\right)\nabla_{x_i}\left(x_i^3\sigma(x_i^1 \cdot s_1 + x_i^2)\right) - s_2\nabla_{x_i}\left(x_i^3\sigma(x_i^1 \cdot s_1 + x_i^2)\right) = \\ &= \nabla_{x_i}\left(-s_2x_i^3\sigma(x_i^1 \cdot s_1 + x_i^2)\right) + \nabla_{x_i}\left(\frac 1N\sum_{j = 1}^Nx_i^3x_j^3\sigma(x_i^1 \cdot s_1 + x_i^2)\sigma(x_j^1 \cdot s_1 + x_j^2)\right) \eqqcolon \\ &\eqqcolon \nabla_{x_i}V_s(x_i) + \nabla_{x_i}\left(\frac 1N\sum_{j=1}^NW_s(x_i, x_j)\right)
\end{aligned}\end{equation*}
where $V_s(x_i) \coloneqq -s_2x_i^3\sigma(x_i^1 \cdot s_1 + x_i^2)$ and $W_s(x_i, x_j) = W_s(x_j, x_i) \coloneqq x_i^3x_j^3\sigma(x_i^1 \cdot s_1 + x_i^2)\sigma(x_j^1 \cdot s_1 + x_j^2)$.
With this choice of $f$, equation \eqref{fund_eq} reads:
\begin{equation*}\frac{\partial \rho^N}{\partial t} = \sum_{i = 1}^N \nabla_{x^i}\cdot \left(\rho^N\nabla_{x^i}V_s(x^i) + \frac 1N\sum_{j = 1}^N\nabla_{x^i}W_s(x^i, x^j)\rho^N\right) + K\int_S(\rho^N(x, s') - \rho^N(x, s))\,\dd\mu(s'). \end{equation*}
We define also the hierarchy of marginals $\rho^{N : k} \coloneqq \int \rho^N(\dd x_{k + 1}, \dots, \dd x_N)$, which satisfy the equations:
\begin{equation*}\begin{aligned} \frac{\partial \rho^{N:k}}{\partial t} &= \sum_{i = 1}^k\nabla_{x_i}\cdot \Bigg(\rho^{N:k}\nabla V_s(x_i) + \frac 1N \sum_{j = 1}^k\nabla W_s(x_i, x_j)\rho^{N:k} + \\ &+ \frac 1N\sum_{j = k + 1}^N \int_{\left(\R^{d + 2}\right)^{N - k}}\nabla W_s(x_i, x_j)\rho^N(\dd x_{k + 1}\dots\dd x_N)\Bigg) + K\int_S (\rho^{N:k}(s') - \rho^{N:k}(s))\,\dd\mu(s'),
\end{aligned}\end{equation*}
for $k = 1, \dots, N$. If we suppose that $\rho_0$ is exchangeable in $x$, the previous system is exchangeable in $x$ as well and it can interpreted in an easier way:
\begin{equation}\begin{aligned} \frac{\partial \rho^{N:k}}{\partial t} &= \sum_{i = 1}^k\nabla_{x_i}\cdot \Bigg(\rho^{N:k}\nabla V_s(x_i) + \frac 1N\sum_{j = 1}^k \nabla W_s(x_i, x_j)\rho^{N:k} + \\ &+ \frac {N-k}N\int_{\R^{d + 2}}\nabla W_s(x_i, x_{k + 1})\rho^{N : k + 1}(\dd x_{k + 1})\Bigg) + K\int_S (\rho^{N:k}(s') - \rho^{N:k}(s))\,\dd\mu(s'), \label{bbgky}
\end{aligned}\end{equation}
for $k = 1, \dots, N$. 
In particular, for a single parameter $x$, equation \eqref{bbgky} reads:
\begin{equation}\begin{aligned} \frac{\partial \rho^1}{\partial t} &= \nabla_{x_1}\cdot \Bigg(\rho^1\nabla V_s(x_1) + \frac 1N\nabla W_s(x_1, x_1)\rho^1 + \\ &+ \frac {N-1}N\int_{\R^{d + 2}}\nabla W_s(x_1, x_2)\rho^{N : 2}(\dd x_2)\Bigg) + K\int_S (\rho^1(s') - \rho^1(s))\,\dd\mu(s'). \label{bbgky_1}
\end{aligned}\end{equation}
If $N \to \infty$, since $\rho^{N:2}$ has two identical marginals on $\R^{d + 2}$, coinciding with $\bar \rho^1$, we can guess that equation \eqref{bbgky_1} converges to the Vlasov equation: 
\begin{equation} \partial_t \rho = \nabla_x \cdot \left(\rho \left(\nabla V_s(x) + \int_{\R^{d + 2}}\nabla W_s(x, y)\,\bar \rho(d y)\right)\right) + K\int_S(\rho(s') - \rho(s))\,\dd\mu(s'). \end{equation}

\section*{Acknowledgments}

The authors gratefully acknowledge the support by the RTG 2339
“Interfaces, Complex Structures and Singular Limits” of the German Research Foundation (DFG).

%%%%%%%%%%%%%%%%%%%%%%%%%%%%%%%%%%%%%%%%%%%%%%%%%%%%%%%%%%%%%%%%%%%%%%%%%%%%%%%%%%%%%%%%%%%%%%%

%%%%%%%%%%%%%%%%%%%%%%%%%%%%%%%%%%%%%%%%%%%%%%%
\bigskip
$^1$\,E-mail: \href{martin.burger@fau.de}{\nolinkurl{martin.burger@fau.de}}\\
$^2$\,E-mail: \href{alex.rossi@fau.de}{\nolinkurl{alex.rossi@fau.de}}

\end{document}